\documentclass[11pt]{article}

\usepackage{amsmath,amssymb,amsthm,mathtools}
\usepackage{lineno}
\usepackage{hyperref}
\usepackage{setspace}
\newcommand{\Feight} {\fontsize{8}{11}\selectfont  }
\newcommand{\Fnine} {\fontsize{9}{11}\selectfont  }

\newtheorem{theorem}{Theorem}[section]
\newtheorem{lemma}[theorem]{Lemma}
\newtheorem{proposition}[theorem]{Proposition}

\newtheorem{assumption}[theorem]{Assumption}
\newtheorem{definition}[theorem]{Definition}
\newtheorem{remark}{Remark}[section]

\title{Weak Information Geometry:\\
Riemannian Structures from Distributional Inference Functions and Stein Discrepancies}
\author{R. Labouriau\footnote{Department of Mathematics, Aarhus University.\\
                                   e-mail:  
                                                 \texttt{rodrigo.labouriau@rlstatlab.com}}
}

\date{Summer 2026}

\begin{document}
\maketitle


\begin{abstract} \Feight
The class of parametric statistical models that can be treated as
Riemannian manifolds---and hence handed to the differential-geometric
theory of statistical models---is considerably larger than the
classical Fisher--Rao setting allows, once one works in the space of
tempered distributions. A law is represented by a tempered
distribution $T_\theta\in\mathcal S'(\mathbb R^k)$, while an
\emph{instrument}---a positive Schwartz kernel, a weak regular
inference function, or a weak Stein representation, in the sense of a
companion paper on inference functionals, generalising the classical
theory of estimating functions---extracts information from the law
without being part of it. Any instrument with full-rank sensitivity
and positive-definite variability induces the Godambe information
$G(\theta)=S(\theta)^\top V(\theta)^{-1}S(\theta)$, a Riemannian
metric on the parameter space; the
Fisher--Rao manifold is recovered exactly when the score is an
admissible instrument, and every Godambe metric is dominated by the
Fisher metric in the Loewner order whenever the latter exists. Four
examples lie outside the
Fisher--Rao class for four different reasons: a location model built
on the Cantor distribution (an undominated family---no likelihood, no
score, and no Fisher information exist at all), the uniform scale
model (parameter-dependent support), the shifted exponential model
(transform-based inference), and a stratified finite mixture (a
provably biased score in a dominated model); a
lattice stochastic heat equation driven by $\alpha$-stable noise
provides a fifth, dynamical example, whose closed-form weak Godambe
information stabilises at a rate governed by the spectral gap of the
discrete Laplacian.
Quadratic Stein discrepancies induce the same local geometry, and
reproducing-kernel constructions generate a hierarchy of geometries.
Because there is no canonical instrument---Chentsov--Markov
invariance is traded for existence---the model carries a family
of Godambe metrics; we discuss the inferential, diagnostic,
geometric, and computational roles of its members, and show that weak
inferential separation (nonformation) appears geometrically as
block-diagonality of the Godambe metric with respect to the
interest--nuisance splitting.
\end{abstract}

\newpage
\Fnine
\tableofcontents
\newpage
\normalsize

\section{Introduction}\label{sec:introduction}

The central message of this paper is that the class of parametric
statistical models that can be viewed as Riemannian manifolds---and
therefore treated with the classical differential-geometric theory of
statistical models---can be \emph{significantly extended} by changing
the way a model is represented. Instead of describing each law by a
probability density, we describe it by the tempered distribution it
induces, that is, by the continuous linear functional that maps a
test function to its expectation; and we extract information from
this representation by pairing it with suitably chosen functions,
which thus play the role of \emph{measurement instruments}. The
geometry of the model is then read off from what the instruments
measure, rather than from the local behaviour of a likelihood. A
single example conveys the reach of the extension. Let $C$ have the
Cantor distribution, and consider the location family generated by
$X = \theta + C$ (see Section~\ref{ssec:cantor} for details). Every
law in this family is singular continuous, and---more radically---the
family admits no dominating $\sigma$-finite measure whatsoever
(Proposition~\ref{prop:cantor-undominated}): there is no likelihood,
no score, no Fisher information, and the Fisher--Rao construction
cannot even begin. Yet the model has moments of every order, an
explicit characteristic function, and, once probed by bounded
instruments, closed-form Riemannian metrics. Nothing about the model
is informationally pathological; what fails is a particular
\emph{representation} of information, the one that insists on
densities.

In the representation adopted throughout, a probability law is thus
represented by a tempered distribution $T_\theta\in\mathcal S'(\mathbb
R^k)$, on the same footing as a density or a characteristic function,
and information is extracted from it through an \emph{instrument}
that serves as a measurement device rather than as part of the law:
in the simplest case, a positive Schwartz kernel, which renders
otherwise inadmissible probes (monomials, oscillators) available;
more generally, one of the two classes of estimating-function-type
instruments described later in this introduction and formally
introduced in Section~\ref{sec:preliminaries}. This construction
simultaneously generalises the classical notion of regular inference
function, due to Godambe \cite{Godambe1960}, and the classical notion
of Stein representation, and allows the methods of
information geometry, as developed by Barndorff-Nielsen
\cite{BarndorffNielsen1978}, Amari \cite{Amari1985}, Amari and Nagaoka
\cite{AmariNagaoka2000}, and others \cite{BarndorffNielsenCoxReid1986},
to be applied in contexts that were otherwise inaccessible, including
models without densities and without dominating measures, models with
parameter-dependent support, models defined only through transforms,
and dominated models whose score function fails to be a sufficiently
regular inference function.

Recall first the classical picture, and where it fails. Under
appropriate regularity conditions, a parametric statistical
model $\mathcal P = \{P_\theta : \theta \in \Theta\}$ carries the
structure of a Riemannian manifold, with the Fisher information matrix
$I(\theta)$, defined by
\[
I(\theta)_{jk}
= \mathrm{E}_\theta\!\left[
    \frac{\partial \log p_\theta(X)}{\partial \theta_j}
    \cdot
    \frac{\partial \log p_\theta(X)}{\partial \theta_k}
\right],
\]
serving as the metric tensor on the parameter
space~$\Theta \subseteq \mathbb R^p$. This
construction requires: (a) the existence of a density
$p_\theta = dP_\theta/d\lambda$ with respect to a dominating
measure~$\lambda$; (b) differentiability of $\log p_\theta$ with
respect to~$\theta$; and (c) finiteness and positive definiteness of
$I(\theta)$.

These conditions fail in many cases of interest. Heavy-tailed
families such as the Cauchy possess no moments of any order, so that
every moment-based route to a geometry is closed there, even though
the likelihood route itself remains regular
(Remark~\ref{rem:cauchy-sinusoid});
singular models in the sense of Watanabe
\cite{Watanabe2009} have degenerate Fisher matrices (see also
Watanabe and Amari \cite{WatanabeAmari2003} for an earlier treatment
of singularities in the information-geometric setting); models defined
only through characteristic functions or other transforms may not
possess densities at all---and singular-continuous families, such as
the Cantor location model of Section~\ref{ssec:cantor}, may admit no
dominating $\sigma$-finite measure whatsoever, so that no likelihood
can even be written down; and, as shown in \cite{Labouriau2022}, there
are finite mixture models---stratified models in which the group
membership of some observations is known only up to probabilities not
in $\{0,1\}$---whose score function is biased when viewed as an
inference function, so that the likelihood-based route to a
Riemannian metric fails even in dominated models with a formally
defined likelihood. In all these situations the classical
Fisher--Rao construction is either unavailable or formally defined but
inferentially meaningless, and the differential-geometric machinery of
information geometry cannot be brought to bear.

The extension rests on two complementary classes of instrument, both
acting on the tempered distribution $T_\theta$ that represents the law,
and both formulated in the space of tempered distributions
$\mathcal{S}'(\mathbb R^k)$:

\begin{enumerate}
\item \emph{Weak regular inference functions.} This is a distributional
generalisation of Godambe's \cite{Godambe1960} classical notion of a
regular inference function and of the formulation systematised in
Chapter~4 of J\o{}rgensen and Labouriau \cite{JorgensenLabouriau2012}.
Regularity is required only in a weak sense: unbiasedness and
differentiability are expressed as pairings between the distributional
representation of the model and suitable test functions, rather than as
pointwise or $P_\theta$-a.s.\ identities on a density. Several concrete
classes of weak regular inference functions, together with the
observation operators that generalise them, are developed in the
companion paper \cite{C}; we recall what is needed here in
Section~\ref{sec:preliminaries}.

\item \emph{Weak Stein representations.} Classical Stein operators
characterise a law via identities of the form
$\mathrm{E}_\theta[\mathcal A_\theta g(X)] = 0$ for $g$ in a test class.
In the distributional framework these are formulated as pairings
$\langle T_\theta, \mathcal A_\theta g\rangle = 0$ with $g \in
\mathcal S(\mathbb R^k)$, which do not require the existence of a
density. Each test function $g$ then yields a weak regular inference
function $\psi_g(x;\theta) = \mathcal A_\theta g(x)$, so the two
instruments are of the same kind.
\end{enumerate}
Both give rise, via sensitivity and variability, to a Godambe-type
information matrix $G(\theta)$ which plays the role of a Riemannian
metric on $\Theta$.

We state the main contribution of the paper as a theorem. Its proof is
short given the constructions developed below: the analytic core is
Proposition~\ref{prop:godambe-metric}
(Section~\ref{sec:godambe-manifold}), and the recovery of the
Fisher--Rao manifold is Section~\ref{ssec:embedding}; the substantive
content is that instruments of the required kind exist for models far
outside the Fisher--Rao class, which is the object of
Proposition~\ref{prop:instruments} and of the examples of
Sections~\ref{sec:examples}--\ref{sec:spde}.

\begin{theorem}[Extension of the class of Riemannian statistical models]\label{thm:main}
Let $\mathcal P = \{P_\theta : \theta \in \Theta\}$, with
$\Theta \subseteq \mathbb R^p$ open, be a parametric family whose laws
are represented by tempered distributions $T_\theta \in
\mathcal S'(\mathbb R^k)$ in the sense of
Section~\ref{ssec:distr-rep}. Suppose the model is equipped with an
\emph{instrument}: a family of weak regular inference functions
$\psi(\cdot,\theta)$---admissible and identifying, and possibly
constructed from a weak Stein representation---whose sensitivity
$S(\theta) = -\mathrm{E}_\theta[\partial_\theta \psi(X,\theta)]$ and
variability $V(\theta) =
\mathrm{E}_\theta[\psi(X,\theta)\psi(X,\theta)^\top]$ are smooth in
$\theta$, with $S(\theta)$ nonsingular and $V(\theta)$ positive
definite. Then the Godambe information
\[
G(\theta) = S(\theta)^\top V(\theta)^{-1} S(\theta)
\]
is a smooth Riemannian metric on $\Theta$; we call the pair
$(\Theta,G)$---the parameter space endowed with the Godambe
information metric---a \emph{Godambe--Riemannian manifold}. When a
density exists and the score
is itself such an instrument, $G(\theta) = I(\theta)$ and the classical
Fisher--Rao manifold is recovered. For a general instrument, whenever
the Fisher information $I(\theta)$ exists and the Bartlett-type
interchange conditions hold, one has the Loewner-order comparison
\[
G(\theta) \;\preceq\; I(\theta),
\]
with equality if and only if $\psi$ is a nonsingular linear
transformation of the score function. In one-parameter models the
comparison reads
\[
G(\theta) = \mathrm{ARE}(\psi,\theta)\,I(\theta),
\qquad \mathrm{ARE}(\psi,\theta)\in(0,1] :
\]
the two metrics are conformally related, and the conformal factor is
the asymptotic relative efficiency.
\end{theorem}

\begin{proof}
Positive-definiteness and smoothness of $G$ are the content of
Proposition~\ref{prop:godambe-metric}. The reduction $G=I$ under the
score and the Loewner comparison for a general instrument are
established in Section~\ref{ssec:embedding}
(Proposition~\ref{prop:loewner}).
\end{proof}

\begin{proposition}[Instruments exist beyond the Fisher--Rao class]\label{prop:instruments}
For each of the following obstructions to the Fisher--Rao construction
there is a model exhibiting it that nonetheless admits an instrument in
the sense of Theorem~\ref{thm:main}, and hence carries a
Godambe--Riemannian structure:
\begin{enumerate}
\item parameter-dependent support---the uniform scale and shifted
exponential models
(Sections~\ref{ssec:uniform}--\ref{ssec:shifted-exp});
\item absence of a density---indeed, absence of \emph{any} dominating
$\sigma$-finite measure, so that no likelihood exists---the Cantor
location model (Section~\ref{ssec:cantor});
\item absence of ordinary moments---Student $t$ with $\nu\le 1$
(Remark~\ref{rem:cauchy-sinusoid}) and the
$\alpha$-stable lattice model, with $\alpha < 2$, of
Section~\ref{sec:spde};
\item a dominated model whose score is a biased inference
function---stratified finite mixed models in the sense of
\cite{Labouriau2022} (Section~\ref{ssec:mixture}).
\end{enumerate}
Consequently the class of Godambe--Riemannian models strictly contains
the Fisher--Rao class.
\end{proposition}

\begin{proof}
The required instruments are exhibited, with their sensitivity and
variability computed in closed form, in the sections and remark
cited.
\end{proof}

The contribution is threefold. \emph{Analytically}, given
Proposition~\ref{prop:godambe-metric}, Theorem~\ref{thm:main} turns
the distributional constructions of \cite{C,JorgensenLabouriau2012}
into a Riemannian structure, and Proposition~\ref{prop:loewner}
places the resulting family of metrics below the Fisher metric in the
Loewner order whenever the latter exists, with the score attaining
the top. \emph{Constructively},
Proposition~\ref{prop:instruments} and the examples of
Sections~\ref{sec:examples}--\ref{sec:spde} exhibit instruments, with
sensitivities and variabilities in closed form, for four
qualitatively distinct obstructions to the classical theory,
culminating in a dynamical example whose geometry stabilises at the
spectral-gap rate of the underlying lattice heat flow.
\emph{Structurally}, Sections~\ref{sec:stein}--\ref{sec:nonformation}
show that quadratic Stein discrepancies read the same local geometry,
that the non-uniqueness of the metric is a trade of Chentsov--Markov
invariance for existence, and that inferential separation appears as
block-diagonality of the metric.

The hypotheses of Theorem~\ref{thm:main}---full-rank sensitivity and
positive-definite variability---could conceivably fail for an
unfortunate choice of instrument. The companion paper
\cite{LabouriauTransversality} shows that such failures are
exceptional in a precise sense: within natural families of
instruments, those violating the hypotheses form a negligible set,
and a small perturbation of a degenerate instrument restores
non-degeneracy---much as two curves drawn at random in the plane may
well cross, but are not expected to be mutually tangent. Statements
of this kind (\emph{genericity}) are the province of transversality
theory in differential topology; none of that machinery is needed in
the present paper, where every instrument is exhibited explicitly and
its non-degeneracy is verified by direct computation.

Several general arguments already illustrate the theorem at this level of
generality. First, bounded weak inference functions such as the
sinusoidal functions $\psi_c(x,\theta) = \sin(c(x - \theta))$ have
finite variability for \emph{any} probability measure, since
$|\sin(\cdot)| \leq 1$. Their sensitivity
$S(\theta) = c \cdot \mathrm{Re}\!\bigl[e^{-ic\theta}\phi_\theta(c)\bigr]$
---the real part of the characteristic function of the \emph{centred}
variable, equal to $c\,\phi_0(c)$ in a symmetric location family---is
likewise always finite. Positive-definiteness of $G(\theta)$ requires
only that this quantity be nonzero for the chosen~$c$; since
$\phi_\theta$ is continuous with $\phi_\theta(0)=1$, there exists
$c_0>0$ such that $G_c(\theta)>0$ for all $0<|c|<c_0$. (Non-vanishing
for \emph{all} $c$ cannot be expected: characteristic functions of
P\'olya type vanish identically outside a compact interval.)

Second, the Godambe information inherits the property that makes the
Fisher information a natural Riemannian metric: it arises as the
variability of an optimal unbiased estimating equation. Godambe's
\cite{Godambe1960} optimality theorem plays the role that the
Cram\'er--Rao bound plays in the classical theory.

Third, when a density exists and the Fisher information is well
defined, every Godambe metric is dominated by the Fisher metric in
the Loewner order, with equality precisely for the score
(Proposition~\ref{prop:loewner}): the Fisher--Rao manifold sits at
the top of the family of Godambe--Riemannian structures whenever it
exists at all.

A Riemannian metric on a statistical model, once constructed, plays
several roles at
once---\emph{inferential} (local distinguishability, efficiency
bounds), \emph{diagnostic} (identifiability, sensitivity analysis),
\emph{geometric} (geodesics, curvature, distances between models),
and \emph{computational} (natural-gradient and preconditioned
algorithms)---and these roles need not coincide outside the
Fisher--Rao setting. The distributional construction yields a whole
family of metrics on the same model, whose members emphasise
different roles; Section~\ref{sec:canonicity} discusses this in
detail.

Several points of contact with existing work should be made explicit,
because the instrument reading settles them. First, when a density $f$
exists the weak moment $\langle T_\theta, x^r\varphi\rangle$ equals an
ordinary moment of the tilted function $\varphi f$, and it is tempting
to dismiss the construction as ``just $\varphi f$''. The reading adopted
here shows why this misses the point: $\varphi$ is an instrument applied
to a fixed law $T_\theta$, not a new law, so---unlike mollification,
which replaces $f$ by a smoothed density and thereby changes the
model---the law is never altered, and the construction continues to make
sense when no density exists at all. Second, sinusoidal and
characteristic-function inference functions coincide with the estimating
equations of empirical-characteristic-function inference
\cite{Heathcote1977,FeuervergerMcDunnough1981} and of transform-based
GMM, including its continuum version \cite{CarrascoFlorens2000}; what
is new here is not those estimating equations but their geometric
reading as instruments inducing a metric, and their use in models with
no density and no moments. Third, geometry generated by estimating
functions has semiparametric precedents in Amari and Kawanabe
\cite{AmariKawanabe1997} and in the estimating-function theory of
\cite{Labouriau1996}; those theories are density-based and
score-centred, whereas the present construction requires neither a
density nor a score, at the price of giving up Chentsov--Markov
invariance (see Section~\ref{sec:canonicity} and
\cite{Chentsov1982,AyJostLeSchwachhofer2017}). Finally, classical
information geometry \cite{Amari1985,BarndorffNielsen1978,AmariNagaoka2000}
is the special case in which the instrument is trivial and the score is
admissible: then $G(\theta)=I(\theta)$, as recorded in
Section~\ref{ssec:embedding}.

The paper is organised as follows. Section~\ref{sec:preliminaries}
recalls the key concepts from the distributional framework: distributional
representations, regular inference functions, weak Stein operators, and
distributional moments. Section~\ref{sec:godambe-manifold}
establishes the general framework, proving that the Godambe information
defines a Riemannian metric under natural regularity conditions and
discussing the embedding of the classical Fisher--Rao theory.
Section~\ref{sec:examples} develops four detailed examples: the uniform
scale model, the shifted exponential model, a location model built on
the Cantor distribution (an undominated family), and a stratified
finite mixture model in the sense of \cite{Labouriau2022}. Section~\ref{sec:spde} treats a structurally richer
example coming from a discrete stochastic heat equation driven by
$\alpha$-stable noise, in which the weak Godambe information can be
computed in closed form and in which the stabilisation rate of the
Godambe metric is shown to be controlled by the spectral gap of the
discrete Laplacian, providing a concrete bridge between the
\emph{dynamical} stability of the underlying system and the
\emph{geometric} stability of the associated statistical model.
Section~\ref{sec:stein} investigates the connection with
Stein discrepancies, establishing an equivalence between quadratic Stein
discrepancies and Godambe geometry, and extending the analysis to
reproducing kernel Hilbert space constructions.
Section~\ref{sec:canonicity} discusses the non-uniqueness of the Godambe
metric and strategies for selecting a canonical geometry.
Section~\ref{sec:nonformation} explores the interaction between
Godambe--Riemannian structures and weak inferential separation,
showing that the off-diagonal block of the Godambe metric measures
the failure of inferential separation, that the Bhapkar--Godambe
projection block-diagonalises the metric under an explicit generating
condition, and that odd/even instrument pairs in symmetric
location-scale models achieve exact separation automatically.
Section~\ref{sec:discussion} summarises the contributions and outlines
directions for future work.

\section{Preliminaries: the distributional framework}\label{sec:preliminaries}

We briefly recall the key concepts from the distributional framework
developed in \cite{A,B,C} and \cite{JorgensenLabouriau2012},
to the extent needed for the present paper. Full details, proofs, and
further developments may be found in those references.

\subsection{Distributional representations}\label{ssec:distr-rep}

Let $\mathcal P = \{P_\theta : \theta \in \Theta\}$ be a parametric
family of probability measures on $(\mathbb R^k, \mathcal B)$, where
$\Theta \subseteq \mathbb R^p$ is an open set. In the distributional
framework the law $P_\theta$ is represented by a \emph{tempered
distribution}: a continuous linear functional
\[
T_\theta : \mathcal S(\mathbb R^k) \longrightarrow \mathbb R,
\qquad\text{that is,}\qquad
T_\theta \in \mathcal S'(\mathbb R^k),
\]
where $\mathcal S(\mathbb R^k)$ denotes the Schwartz space of
infinitely differentiable functions all of whose derivatives decay
faster than any polynomial, and $\mathcal S'(\mathbb R^k)$, the space
of \emph{tempered distributions}, is its topological dual. Following
the usage of functional analysis, the evaluation of a distribution
$T$ at a test function $g$ is written as a \emph{pairing},
\[
\langle T, g \rangle \;:=\; T(g),
\]
a notation used throughout the paper. The representation is defined
by the requirement that, for all $g \in \mathcal S(\mathbb R^k)$,
\[
\mathrm{E}_{P_\theta}[g(X)]
\;=\; \int_{\mathbb R^k} g(x)\, dP_\theta(x)
\;=\; \langle T_\theta, g \rangle ;
\]
the integral is always well defined, and finite, because a Schwartz
function is bounded and $P_\theta$ is a probability measure.
(Complex-valued test functions, needed for oscillatory probes, are
handled by pairing real and imaginary parts separately.) Thus
$T_\theta$ is nothing more exotic than the expectation operator of
$P_\theta$ restricted to Schwartz test functions, and it
characterises the law on the same footing as a density, a
distribution function, or a characteristic function.

Information about the law is extracted through an \emph{instrument}: a
positive Schwartz kernel $\varphi \in \mathcal S(\mathbb R^k)$, or more
generally an observation operator in the sense of the companion paper
\cite{C}, that serves as a measurement device and is
\emph{not} part of the law. The instrument renders admissible the probes
one actually wishes to apply---monomials $x^r$, oscillators $e^{iux}$,
the outputs $\mathcal A_\theta g$ of a Stein operator---none of which is
a Schwartz function, and none of which therefore pairs with a general
$T_\theta \in \mathcal S'(\mathbb R^k)$ on its own; multiplication by
$\varphi$ places $x^r\varphi,\, e^{iux}\varphi \in \mathcal S(\mathbb
R^k)$, so that the \emph{weak expectation}
\[
{}^{(\varphi)}\mathrm{E}_\theta[g] := \langle T_\theta, g\varphi\rangle
\]
is well defined.

Because $T_\theta$ represents a \emph{probability measure}, the pairing
extends canonically well beyond $\mathcal S(\mathbb R^k)$. We record
this once, so that all the instruments used in this paper---sinusoids,
cosines of linear functionals, characteristic-function residuals,
polynomials---are admissible without case-by-case cutoff arguments.

\begin{lemma}[Extension of the pairing]\label{lem:extension}
Let $T_\theta\in\mathcal S'(\mathbb R^k)$ represent $P_\theta$, i.e.\
$\langle T_\theta,g\rangle = \int g\,dP_\theta$ for all
$g\in\mathcal S(\mathbb R^k)$. Then the pairing extends uniquely to
\[
\langle T_\theta, g\rangle \;:=\; \int g\, dP_\theta,
\qquad g \in L^1(P_\theta),
\]
the extension being continuous under dominated pointwise convergence
and consistent with the distributional pairing whenever the latter is
defined. In particular the following \emph{admissible classes} of
probes may be paired with $T_\theta$:
\begin{enumerate}
\item $\mathcal G_{\mathrm b}$: bounded continuous functions
(sinusoids, cosines of linear functionals, residuals
$e^{itx}-\phi_\theta(t)$, constants);
\item $\mathcal G_{\varphi} := \{g : g\varphi\in\mathcal S(\mathbb
R^k)\}$ for a positive Schwartz kernel $\varphi$: kernel-weighted
probes (monomials $x^r$ and oscillators $e^{iux}$ observed through
$\varphi$, as in the weak expectation
${}^{(\varphi)}\mathrm{E}_\theta$);
\item $\mathcal G_{\mathrm{mom}}$: measurable functions of polynomial
growth, admissible at those $\theta$ for which the corresponding
ordinary moments of $P_\theta$ exist.
\end{enumerate}
\end{lemma}

\begin{proof}
\emph{Existence} is immediate: the map $g\mapsto\int g\,dP_\theta$ is
well defined on $L^1(P_\theta) \supseteq \mathcal S(\mathbb R^k)$ and
restricts to the distributional pairing on $\mathcal S(\mathbb R^k)$
by the representation property.

\emph{Uniqueness.} Let $\Lambda$ be any extension to $L^1(P_\theta)$
that is continuous under dominated pointwise sequential convergence
and agrees with the pairing on $\mathcal S(\mathbb R^k)$, and set
$\mathcal C := \{g \in L^1(P_\theta) : \Lambda(g) = \int
g\,dP_\theta\}$. By the continuity assumption (applied to both
$\Lambda$ and the integral), $\mathcal C$ is closed under dominated
pointwise sequential limits, and $\mathcal S(\mathbb R^k) \subseteq
\mathcal C$ by hypothesis.

\emph{Step 1: $C_b(\mathbb R^k) \subseteq \mathcal C$.} Given $g \in
C_b(\mathbb R^k)$, choose smooth cutoffs $\chi_n$ with $\mathbf
1_{B(0,n)} \le \chi_n \le \mathbf 1_{B(0,n+1)}$ and mollifiers
$\rho_n$ supported in $B(0,1/n)$, and set $g_n := (g\chi_n) * \rho_n
\in C_c^\infty(\mathbb R^k) \subset \mathcal S(\mathbb R^k)$. Then
$\|g_n\|_\infty \le \|g\|_\infty$ (convolution of a function bounded
by $\|g\|_\infty$ with a probability density), and $g_n(x) \to g(x)$
for every $x$: for $n$ large one has $g_n = g * \rho_n$ on a
neighbourhood of $x$, and $g * \rho_n(x) \to g(x)$ by continuity of
$g$. Hence $g \in \mathcal C$, with dominating function the constant
$\|g\|_\infty$.

\emph{Step 2: every bounded Borel function belongs to $\mathcal C$.}
The collection $\mathcal H := \{g \text{ bounded Borel} : g \in
\mathcal C\}$ is a vector space containing the constants, containing
the algebra $C_b(\mathbb R^k)$, which separates the points of
$\mathbb R^k$, and closed under bounded pointwise sequential limits
(such limits are dominated by a constant). By the functional monotone
class theorem, $\mathcal H$ contains every bounded function
measurable with respect to $\sigma(C_b(\mathbb R^k)) = \mathcal
B(\mathbb R^k)$.

\emph{Step 3: from bounded to integrable.} For $g \in L^1(P_\theta)$,
the truncations $g\,\mathbf 1\{|g| \le n\}$ are bounded Borel
functions converging pointwise to $g$ and dominated by $|g| \in
L^1(P_\theta)$; hence $g \in \mathcal C$.

Finally, for $g\in\mathcal G_\varphi$ the quantity $\langle T_\theta,
g\varphi\rangle$ is defined distributionally and equals $\int
g\varphi\,dP_\theta$ by the representation property, consistently
with the extension.
\end{proof}

Throughout the paper, every pairing $\langle T_\theta,
\psi(\cdot;\theta)\rangle$ is understood in the sense of
Lemma~\ref{lem:extension}, and each example names the admissible class
in which its instrument lives.

When a density $p_\theta = dP_\theta/d\lambda$ exists, $T_\theta$ is the
regular distribution defined by $p_\theta$ and the pairing reduces to
$\langle T_\theta, g\rangle = \int g\,p_\theta\,d\lambda$; the framework
then contains the classical one. Its advantage is that it applies
verbatim to models without densities, including models defined through
characteristic-function constraints, distributional equations, or
transform-based specifications.

\begin{remark}[Singularity as differentiated regularity]
The use of tempered distributions does not introduce arbitrarily
pathological objects. By the classical structure theorem (Strichartz,
\emph{A Guide to Distribution Theory and Fourier Transforms},
Section~6.3), every tempered distribution is a finite sum of derivatives
of continuous functions of at most polynomial growth; singular
probabilistic behaviour---point masses, jumps, heavy tails---arises when
ordinary functions are differentiated in the weak sense. The instrument
$\varphi$ acts as a regularising observational device that converts
these differentiated structures into stable scalar quantities.
\end{remark}

\subsection{Weak regular inference functions}\label{ssec:rif-def}

The inference-function apparatus used throughout the paper is a
\emph{weak} version, in the distributional sense, of Godambe's
\cite{Godambe1960} classical notion of a regular inference function
and of the systematic formulation of regular inference functions
developed in Chapter~4 of J\o{}rgensen and Labouriau
\cite{JorgensenLabouriau2012}. The adjective ``weak'' refers to the
fact that unbiasedness and differentiability conditions are expressed
as pairings between a distributional representation of the model and
elements of a test-function class, rather than as pointwise or
$P_\theta$-almost-sure identities involving a density. When a density
exists and is sufficiently regular, the weak definitions reduce to
the classical ones; when a density does not exist, or the score
function fails to be a sufficiently regular inference function, the weak
definitions nevertheless make sense. The systematic development of
weak regular inference functions is carried out in the companion
paper \cite{C}; here we record only what is needed.

\begin{definition}[Weak regular inference function]\label{def:rif}
A measurable function
$\psi : \mathbb R^k \times \Theta \to \mathbb R^q$
is called a \emph{weak regular inference function} for the model
$\mathcal P$ if the following conditions hold:
\begin{enumerate}
    \item \textbf{Admissibility:} For each $\theta \in \Theta$,
    $\psi(\cdot,\theta)$ belongs to one of the admissible classes of
    Lemma~\ref{lem:extension} (or to a class on which the extended
    pairing is defined), so that all pairings below are well defined.

    \item \textbf{Weak unbiasedness:} For all $\theta \in \Theta$,
    \[
    \langle T_\theta,\,\psi(\cdot,\theta)\rangle
    \;=\; \mathrm{E}_{P_\theta}[\psi(X,\theta)] \;=\; 0.
    \]

    \item \textbf{Local identification:} For every $\theta_0 \in
    \Theta$ there is a neighbourhood $U \ni \theta_0$ such that, for
    any probability measure $P$ in a suitable class with $\theta(P)
    \in U$ well defined, $\mathrm{E}_P[\psi(X,\theta)] = 0$ with
    $\theta \in U$ implies $\theta = \theta(P)$.
\end{enumerate}
\end{definition}

When the sensitivity $S(\theta)$ defined below is nonsingular, local
identification along the model holds automatically by the implicit
function theorem; global identification typically requires a richer
probe family (for instance several frequencies $t$ in the
transform-based case) and is not needed for the local geometric
constructions of this paper.

\begin{remark}[Reduction to the classical definition]\label{rem:rif-classical}
When $T_\theta$ is the distributional representation of
an absolutely continuous model and $\psi$ is sufficiently regular in
the classical sense, Definition~\ref{def:rif} reduces to the
definition of a regular inference function in the sense of Godambe
\cite{Godambe1960} and J\o{}rgensen and Labouriau
\cite[Ch.~4]{JorgensenLabouriau2012}.
\end{remark}

The standard inference-theoretic quantities associated with a weak
regular inference function are also defined in a weak sense. The
\emph{weak sensitivity} of $\psi$ at $\theta$ is
\[
S(\theta) \;=\; -\,\mathrm{E}_\theta[\partial_\theta \psi(X;\theta)]
           \;=\; -\,\langle T_\theta,\,\partial_\theta\psi(\cdot;\theta)\rangle,
\]
where the derivative $\partial_\theta \psi(\cdot;\theta)$ is
understood as an element of the relevant test-function class and the
pairing is taken in the distributional sense; the \emph{weak
variability} is
\[
V(\theta) \;=\; \mathrm{E}_\theta[\psi(X;\theta)\psi(X;\theta)^\top]
           \;=\; \langle T_\theta,\,\psi(\cdot;\theta)\psi(\cdot;\theta)^\top\rangle;
\]
and the \emph{weak Godambe information} associated with $\psi$ is
\[
G(\theta) \;=\; S(\theta)^\top V(\theta)^{-1} S(\theta).
\]
The regularity in $\theta$ that these definitions require is an
assumption on the \emph{model map}, which we now state once and use
throughout.

\begin{assumption}[Weak differentiability of the model map]
\label{ass:weakC1}
For every admissible probe $g$ (Lemma~\ref{lem:extension}) the map
$\theta \mapsto \langle T_\theta, g\rangle$ is continuously
differentiable on $\Theta$, and there exist
$\partial_{\theta_j} T_\theta$, $j=1,\dots,p$, acting on admissible
probes, such that
$\partial_{\theta_j} \langle T_\theta, g\rangle
= \langle \partial_{\theta_j}T_\theta,\, g\rangle$.
\end{assumption}

Differentiating the weak unbiasedness identity $\langle T_\theta,
\psi(\cdot;\theta)\rangle = 0$ under Assumption~\ref{ass:weakC1} yields
the \emph{weak Bartlett identity}
\begin{equation}\label{eq:weak-bartlett}
S(\theta)
\;=\; -\,\langle T_\theta,\, \partial_\theta\psi(\cdot;\theta)\rangle
\;=\; \langle \partial_\theta T_\theta,\, \psi(\cdot;\theta)\rangle :
\end{equation}
the sensitivity is the pairing of the \emph{model derivative} with the
\emph{instrument}. Identity~\eqref{eq:weak-bartlett} is the analytic
engine of the paper: it justifies the local expansions of
Section~\ref{sec:stein} and underlies the geometric characterisation
of inferential separation in Section~\ref{sec:nonformation}.
In the scalar case ($q = p = 1$) this reduces to
$G(\theta) = S(\theta)^2 / V(\theta)$. When a density exists and
$\psi$ is a classical regular inference function, the weak
sensitivity, variability, and Godambe information coincide with
their classical counterparts as defined in Godambe \cite{Godambe1960}
and J\o{}rgensen and Labouriau \cite[Ch.~4]{JorgensenLabouriau2012};
thus weak Godambe information is a strict generalisation of the
Godambe information used in the theory of optimal estimating
equations. The Godambe information characterises the asymptotic
efficiency of inference based on~$\psi$: Godambe's \cite{Godambe1960}
theorem shows that $G(\theta)$ is maximised when $\psi$ is the score
function, whenever the score function exists and is itself a valid
(weak) regular inference function.

Typical examples of weak regular inference functions, as developed in
\cite{C}, include moment-based inference functions
$\psi(x;\theta) = h(x) - \mathrm{E}_\theta[h(X)]$ for suitable test
functions $h \in \mathcal S(\mathbb R^k)$, sinusoidal inference
functions $\psi_c(x;\theta) = \sin(c(x - \theta))$, and
characteristic-function-based (transform) inference functions
$\psi_t(x;\theta) = e^{itx} - \phi_\theta(t)$. In each case the
inference function lives naturally in the Schwartz class or in a
closely related test-function space, and pairs with the
distributional representation of the model to yield the weak
sensitivity, weak variability, and weak Godambe information required
by Theorem~\ref{thm:main}.

\subsection{Weak Stein operators}\label{ssec:weak-stein}

In the classical setting, a \emph{Stein operator} for a distribution
$P_\theta$ is a linear operator $\mathcal A_\theta$ acting on a class of
test functions such that
\[
\mathrm{E}_\theta[\mathcal A_\theta g(X)] = 0
\]
for all $g$ in the class. For example, the Gaussian location model
$\mathcal N(\theta, \sigma^2)$ has the Stein operator
$\mathcal A_\theta g(x) = g'(x) - \frac{x - \theta}{\sigma^2} g(x)$.
The identity $\mathrm{E}_\theta[\mathcal A_\theta g(X)] = 0$
characterises $P_\theta$: if $Q$ satisfies
$\mathrm{E}_Q[\mathcal A_\theta g(X)] = 0$ for all test functions $g$,
then $Q = P_\theta$.

In the distributional framework, a \emph{weak Stein operator} is defined
analogously, but the characterising identity is required to hold only for
test functions in $\mathcal S(\mathbb R^k)$:
\[
\langle T_\theta, \mathcal A_\theta g \rangle = 0
\qquad \text{for all } g \in \mathcal S(\mathbb R^k).
\]
This formulation does not require the existence of a density. The connection
to inference functions is immediate: for each test function $g$, the map
$\psi_g(x;\theta) = \mathcal A_\theta g(x)$ defines an unbiased inference
function. A \emph{Stein discrepancy} between $Q$ and $P_\theta$ is then
\[
D_{\mathrm S}(Q, P_\theta)
= \sup_{g \in \mathcal G} |\mathrm{E}_Q[\mathcal A_\theta g(X)]|,
\]
or a quadratic variant thereof.

\subsection{Distributional moments}\label{ssec:distr-moments}

In the distributional framework, \emph{weak (distributional) moments}
of all orders are defined for any law admitting a distributional
representation: following the series convention, the weak moment of
order $r$ with kernel $\varphi$ is
\[
{}^{(\varphi)}m_r(\theta) \;:=\; \langle T_\theta, x^r\varphi\rangle,
\]
the kernel being recorded as a parenthesised presuperscript. The weak
moment is \emph{not} the ordinary moment: it is the $r$-th moment of
the law reweighted by the kernel, and it exists for every $r$ and
every $\theta$ regardless of the tail behaviour of $P_\theta$ (in
particular, all weak moments of the Cauchy and Student~$t$ families
are finite). When the ordinary moment $\mathrm{E}_\theta[X^r]$ does
exist, it is recovered from the weak moments as a limit along a
sequence of kernels converging pointwise to the constant function
$1$ on the support of $P_\theta$: if $(\varphi_n)$ are admissible
kernels with $0 \le \varphi_n \le 1$ and $\varphi_n \uparrow 1$
pointwise (for instance Gaussian kernels of increasing width), then
\[
{}^{(\varphi_n)}m_r(\theta)
= \int x^r \varphi_n(x)\, dP_\theta(x)
\;\longrightarrow\; \mathrm{E}_\theta[X^r],
\]
by dominated convergence, with dominating function $|x|^r \in
L^1(P_\theta)$. Full details, including the kernel-governed
moment-determinacy theory, are given in \cite{A}.

\section{The Godambe--Riemannian manifold}\label{sec:godambe-manifold}

\subsection{General proposition}\label{ssec:general-prop}

The following proposition establishes that the Godambe information
defines a Riemannian metric under natural conditions.

\begin{proposition}\label{prop:godambe-metric}
Let $\mathcal P=\{P_\theta:\theta\in\Theta\}$ be an identifiable
parametric statistical model (i.e., the map $\theta \mapsto P_\theta$
is injective),
where $\Theta \subseteq \mathbb R^p$ is an open set. Suppose there exists a
regular inference function $\psi(\cdot;\theta)$ with values in
$\mathbb R^q$, $q \ge p$
(in the sense of Definition~\ref{def:rif}) such that:
\begin{enumerate}
    \item the sensitivity matrix
    $S(\theta)=-\mathbb E_\theta[\partial_\theta \psi(X;\theta)]
    \in \mathbb R^{q\times p}$
    is well defined and smooth in $\theta$;

    \item the variability matrix
    $V(\theta)=\mathbb E_\theta[\psi(X;\theta)\psi(X;\theta)^\top]$
    is well defined, smooth, and positive definite for all $\theta\in\Theta$;

    \item $S(\theta)$ has full column rank $p$ for all
    $\theta\in\Theta$ (when $q=p$, this is non-singularity).
\end{enumerate}
Then the Godambe information matrix
\[
G(\theta)=S(\theta)^\top V(\theta)^{-1}S(\theta)
\]
defines a smooth Riemannian metric on $\Theta$.
\end{proposition}

\begin{proof}
By assumption, $V(\theta)$ is positive definite and $S(\theta)$ has
full column rank
for all $\theta\in\Theta$. Hence, for any non-zero $a\in\mathbb R^p$,
$S(\theta)a \neq 0$ and
\[
a^\top G(\theta)a
=
(S(\theta)a)^\top V(\theta)^{-1}(S(\theta)a)>0.
\]
Thus $G(\theta)$ is positive definite. Smoothness follows from the smoothness of
$S(\theta)$ and $V(\theta)$ and from the smooth dependence of matrix inversion
on positive-definite matrices.
\end{proof}

\subsection{Embedding of the classical Fisher--Rao theory}\label{ssec:embedding}

Suppose in this subsection that a density $p_\theta$ exists, that the
score $u_\theta = \partial_\theta \log p_\theta$ is itself a valid
(weak) regular inference function, and that the Fisher information
$I(\theta) = \mathrm{E}_\theta[u_\theta u_\theta^\top]$ is finite and
positive definite. The relation between the Godambe metrics and the
Fisher metric is then an ordering, not merely a conformal rescaling.

\begin{proposition}[Loewner comparison with the Fisher metric]\label{prop:loewner}
Let $\psi$ satisfy the hypotheses of
Proposition~\ref{prop:godambe-metric}, and assume the Bartlett-type
interchange condition
\begin{equation}\label{eq:bartlett-score}
\mathrm{E}_\theta\bigl[\psi(X;\theta)\, u_\theta(X)^\top\bigr]
\;=\; S(\theta),
\end{equation}
obtained by differentiating the unbiasedness identity through the law.
Then
\[
G_\psi(\theta) \;\preceq\; I(\theta)
\qquad\text{(Loewner order)},
\]
with equality if and only if the score lies $P_\theta$-a.s.\ in the
linear span of the components of $\psi$; when $q=p$ this means
$\psi(X;\theta) = A(\theta)\,u_\theta(X)$ a.s.\ with $A(\theta)$
nonsingular. In particular $G_{u}(\theta) = I(\theta)$: the
Fisher--Rao metric is recovered exactly when the instrument is the
score.
\end{proposition}

\begin{proof}
By \eqref{eq:bartlett-score}, the $L^2(P_\theta)$-projection of the
score onto the closed linear span of the components of $\psi$ is
$\Pi u_\theta = S(\theta)^\top V(\theta)^{-1} \psi(\cdot;\theta)$,
and its covariance matrix is $S^\top V^{-1} V V^{-1} S = G_\psi(\theta)$.
Hence
\[
I(\theta) - G_\psi(\theta)
= \mathrm{E}_\theta\!\left[
   (u_\theta - \Pi u_\theta)(u_\theta - \Pi u_\theta)^\top\right]
\;\succeq\; 0,
\]
with equality if and only if $u_\theta = \Pi u_\theta$
$P_\theta$-a.s.
\end{proof}

In the scalar case ($p=q=1$) the comparison is conformal:
$G_\psi(\theta) = \mathrm{ARE}(\psi, \theta)\cdot I(\theta)$ with
conformal factor the asymptotic relative efficiency
\[
\mathrm{ARE}(\psi, \theta)
= \frac{S(\theta)^2}{V(\theta)\, I(\theta)} \;\le\; 1 .
\]
As an illustration, consider the sinusoidal inference function
$\psi_c(x;\theta) = \sin(c(x-\theta))$ in the Gaussian location model
$\mathcal N(\theta,\sigma^2)$, with $\sigma^2$ known. Since
$\mathrm{E}_\theta[e^{ic(X-\theta)}] = e^{-c^2\sigma^2/2}$,
\begin{align*}
S(\theta) &= c\,\mathrm{E}_\theta[\cos(c(X-\theta))]
           = c\, e^{-c^2\sigma^2/2},\\
V(\theta) &= \tfrac12\bigl(1 - \mathrm{E}_\theta[\cos(2c(X-\theta))]\bigr)
           = \tfrac12\bigl(1 - e^{-2c^2\sigma^2}\bigr),
\end{align*}
whence
\[
G_c(\theta) = \frac{S(\theta)^2}{V(\theta)}
= \frac{2c^2\, e^{-c^2\sigma^2}}{1 - e^{-2c^2\sigma^2}}
= \frac{c^2}{\sinh(c^2\sigma^2)} .
\]
Writing $w = c^2\sigma^2 = c^2/I_F$,
\[
\mathrm{ARE}(c) = \frac{w}{\sinh(w)} \;\uparrow\; 1
\qquad (c \to 0),
\]
in agreement with the location--scale computations of
\cite[Section~8.5]{C}: full Fisher efficiency is
approached in the small-frequency limit. The extension provided by
Theorem~\ref{thm:main} is non-trivial precisely for those models where
the score is unavailable, or fails to be a sufficiently regular
inference function,
so that the top of the Loewner order is vacant.

\section{Examples}\label{sec:examples}

The four examples presented in this section are chosen to illustrate
that the extension of the class of Riemannian statistical models
claimed in Theorem~\ref{thm:main} is genuine---each example lies
outside the reach of the classical Fisher--Rao construction, and
each does so for a \emph{different} structural reason. Together they
delimit four qualitatively distinct obstructions that the
distributional framework is able to handle.

\begin{itemize}
\item \textbf{The uniform scale model
(Section~\ref{ssec:uniform})} is the prototypical example of a model
with \emph{parameter-dependent support}. The Fisher--Rao machinery
fails here at the most basic level: the support of $P_\theta$ changes
with $\theta$, the score function is not unbiased, and the regularity
conditions of classical information geometry are violated from the
start. A simple first-moment inference function, however, is a weak
regular inference function and yields a smooth Godambe--Riemannian
metric on $(0,\infty)$.

\item \textbf{The shifted exponential model
(Section~\ref{ssec:shifted-exp})} also exhibits
parameter-dependent support, but is included for a different reason:
it is a model in which a \emph{transform-based} weak regular
inference function is particularly transparent. The characteristic
function is available in closed form and gives rise to bounded
sinusoidal inference functions whose weak sensitivity and variability
are easy to compute, illustrating how transform-based weak regular
inference functions generate a Godambe--Riemannian structure on a
model that is singular from the Fisher--Rao point of view.

\item \textbf{The Cantor location model
(Section~\ref{ssec:cantor})} removes the very possibility of a
likelihood. The Cantor distribution is singular continuous, so no
observation density exists; more radically, the location family it
generates admits \emph{no dominating $\sigma$-finite measure at all}
(Proposition~\ref{prop:cantor-undominated}), so likelihood ratios,
scores, and Fisher information are not merely irregular but
nonexistent. Bounded transform-based instruments and moment
instruments nevertheless yield closed-form Godambe metrics, and the
centred model is a symmetric location--scale family to which the
automatic-separation result of
Section~\ref{ssec:automatic-orthogonality} applies verbatim.

\item \textbf{The stratified finite mixture
(Section~\ref{ssec:mixture})} is a qualitatively different
obstruction. The model---observations stratified into groups, with the
group membership of some observations known only up to probabilities
not in $\{0,1\}$---is \emph{dominated} by Lebesgue measure and
admits a formally defined, smooth, everywhere positive likelihood, so
the Fisher--Rao construction is not ruled out by non-dominance. As
proved in \cite{Labouriau2022}, however, the score function of the
mixed model is a \emph{biased} inference function whenever a genuine
mixture is present, and the maximum likelihood estimator is then
inconsistent: the likelihood-based route to a Riemannian metric fails
inferentially. A pair of \emph{simultaneously unbiased} moment-based
weak inference functions---unbiased under every component law, hence
under every stratification---recovers a Godambe--Riemannian metric on
the parameter space, illustrating that the extension claimed here
covers even dominated models where the classical construction is
formally available but inferentially defective.
\end{itemize}

In all four cases the construction is the same in spirit: one
exhibits a weak regular inference function (or a weak Stein
representation that generates one), computes the weak sensitivity
and variability, and reads off the Godambe information as a
Riemannian metric tensor on the parameter space.

\begin{remark}[Student $t$ and Cauchy: the moment obstruction]
\label{rem:cauchy-sinusoid}
Beyond the four examples treated in detail, the Student~$t$ families
with small degrees of freedom illustrate the moment obstruction: for
$\nu \le 2$ the variance does not exist, and for $\nu \le 1$ not even
the mean. The sinusoidal Godambe information is well defined for all
$\nu > 0$, and for the Cauchy location family ($\nu = 1$), where
$\phi_0(c) = e^{-|c|}$, it is available in closed form:
$S(\theta) = c\,e^{-|c|}$, $V(\theta) =
\tfrac12\bigl(1-e^{-2|c|}\bigr)$, and
\[
G_c(\theta) \;=\; \frac{2c^2 e^{-2|c|}}{1-e^{-2|c|}},
\]
a flat metric on $\Theta = \mathbb R$, consistent with the
location-scale computation of \cite[Section~8.5]{C}.

It is worth stressing what is, and what is not, obstructed here. The
Cauchy family is \emph{regular in the classical sense}: writing the
density with location $\mu$ and scale $\sigma>0$, the scores
\[
s_\mu(x) = \frac{2(x-\mu)}{\sigma^2+(x-\mu)^2},
\qquad
s_\sigma(x) = \frac{(x-\mu)^2-\sigma^2}{\sigma\bigl[\sigma^2+(x-\mu)^2\bigr]},
\]
are bounded in absolute value by $1/\sigma$---in particular they are
square integrable, and they belong to the admissible class $\mathcal
G_{\mathrm b}$ of Lemma~\ref{lem:extension}---the Fisher information
is finite and positive definite,
\[
I(\mu,\sigma)
=
\mathrm{diag}\bigl(1/(2\sigma^2),\,1/(2\sigma^2)\bigr),
\]
and the maximum likelihood estimator is consistent and asymptotically
normal with variance $I(\theta)^{-1}$.\footnote{The regularity is
genuine, but the likelihood surface is not concave: the score equation
is a polynomial equation of degree $2n-1$, and the number of
extraneous local maxima converges in distribution to a Poisson
variable with mean $1/\pi$ \cite{Reeds1985}, so that the probability
of a unique root tends to $e^{-1/\pi}\approx0.73$ rather than to one.
The asymptotic statement therefore concerns the global maximiser
(equivalently, the consistent root), and a $\sqrt n$-consistent
starting value is needed in practice---the sample median, say, and not
the sample mean, which is not consistent here. In finite samples the
observed information is preferable to $I(\theta)$ as a measure of
accuracy, the configuration statistic being ancillary in this location
model \cite{EfronHinkley1978}.} The obstruction in the Cauchy family
is thus the absence of \emph{moments}, not of the score: it is the
moment-based instruments, and only those, that are unavailable.

Because the Fisher information exists, this family also illustrates
the Loewner comparison of Proposition~\ref{prop:loewner} in a
heavy-tailed setting. At unit scale, $G_c(\theta)\le\tfrac12=I(\theta)$
for every frequency, and the maximal sinusoidal efficiency is
$\sup_c G_c(\theta)/I(\theta)\approx0.65$, attained near $|c|\approx
0.8$; for general $\sigma$ the efficiency depends on $c$ and $\sigma$
only through $\sigma c$, so the optimal frequency is $|c|\approx
0.8/\sigma$. Similarly, in
models with degenerate Fisher information (mixtures at boundary
points, non-identifiable components), a suitably chosen instrument
may retain a nonsingular Godambe matrix, preserving the Riemannian
structure where the Fisher--Rao one collapses.
\end{remark}

\begin{remark}[Closure of the instrument family, suggested by
S.~Zabolotnii]
\label{rem:cauchy-closure}
The deficiency recorded in Remark~\ref{rem:cauchy-sinusoid} is a
property of the \emph{instrument}, not of the framework, and it
disappears as the instrument family is enriched. Take frequencies
$0<c_1<\dots<c_m$ and the vector instrument
$\psi=\bigl(\sin(c_j(x-\mu))\bigr)_{j=1}^{m}$. Its sensitivity and
variability are available in closed form,
\[
S_j = c_j e^{-\sigma c_j},
\qquad
V_{jk} = \tfrac12\Bigl(e^{-\sigma|c_j-c_k|}
                      -e^{-\sigma(c_j+c_k)}\Bigr),
\]
so the relative efficiency $\kappa_m = S^\top V^{-1}S/I(\theta)$ of
the optimally weighted combination is exactly computable. Optimising
over the frequencies at $\sigma=1$ gives
\[
\kappa_1,\dots,\kappa_7
= 0.648,\;0.820,\;0.891,\;0.927,\;0.948,\;0.961,\;0.969,
\]
and $\kappa_m\uparrow1$. The limit is a completeness statement: if $f$
is odd, $f\in L^2(P_\mu)$ and $\mathrm{E}[f(X)\sin(c(X-\mu))]=0$ for
every $c>0$, then $h=f\,p_\mu\in L^1$ has vanishing Fourier
transform---the cosine part by oddness, the sine part by
hypothesis---so $h=0$ and hence $f=0$, the Cauchy density being
strictly positive. The sinusoidal instruments therefore span the odd
subspace of $L^2_0(P_\mu)$ densely, and the Cauchy location score is
odd; by Proposition~\ref{prop:loewner} the Fisher metric is the
supremum of the Godambe metrics over this family, approached but---for
every finite $m$, since no finite sinusoidal combination equals the
score---not attained.

Two qualifications should be kept in mind. The construction is an
oracle one: the optimal frequencies scale with the unknown $\sigma$,
and the optimal weights involve the population quantities $S$ and
$V$. And the Gram matrix becomes ill conditioned as the family grows
($\operatorname{cond}(V)\approx23$, $117$ and $1.4\times10^{3}$ for
equally spaced grids with $m=10$, $40$ and $160$), the ill-posedness
familiar from the continuum-of-moment-conditions theory
\cite{CarrascoFlorens2000}; the phenomenon itself is the
characteristic-function counterpart of the classical efficiency result
for empirical-characteristic-function procedures
\cite{FeuervergerMcDunnough1981}.
\end{remark}

\subsection{The uniform distribution}\label{ssec:uniform}

\begin{proposition}\label{prop:uniform}
Consider the uniform scale model
\[
X \sim \mathrm{Unif}(0,\theta), \qquad \theta>0.
\]
This model is not regular in the classical Fisher sense, since the support depends
on the parameter and the formal score function is not unbiased. However, the
function
\[
\psi(x;\theta)=x-\frac{\theta}{2}
\]
is a regular inference function
(in the sense of Definition~\ref{def:rif}; the instrument lies in the
class $\mathcal G_{\mathrm{mom}}$ of Lemma~\ref{lem:extension}, all
moments of $\mathrm{Unif}(0,\theta)$ being finite).
Its sensitivity and variability are
\[
S(\theta)=\frac12,
\qquad
V(\theta)=\frac{\theta^2}{12},
\]
and the corresponding Godambe information is
\[
G(\theta)=\frac{3}{\theta^2}.
\]
Hence the parameter space $(0,\infty)$ carries the Riemannian metric
\[
g_\theta=\frac{3}{\theta^2}\,d\theta^2.
\]
\end{proposition}

\begin{proof}
Since $\mathbb E_\theta[X]=\theta/2$, we have
$\mathbb E_\theta[\psi(X;\theta)]=0$.
Moreover,
$\partial_\theta \psi(x;\theta)=-\frac12$,
so
$S(\theta) = -\mathbb E_\theta[\partial_\theta \psi(X;\theta)] = \frac12$.
Also,
$V(\theta) = \mathbb E_\theta[\psi(X;\theta)^2] = \mathrm{Var}_\theta(X)
= \frac{\theta^2}{12}$.
Therefore
$G(\theta)=\frac{S(\theta)^2}{V(\theta)}=\frac{3}{\theta^2}$,
which is smooth and strictly positive for all $\theta>0$.
\end{proof}

\begin{remark}\label{rem:uniform-log}
If one introduces the logarithmic parametrisation $\eta=\log\theta$, the metric becomes
$g=3\,d\eta^2$,
so the model is isometric to a Euclidean line up to scale. This illustrates that
the Godambe metric may induce a natural geometric structure even in a model for
which the classical Fisher construction is not available.
\end{remark}

\begin{remark}[Transform-based metric for the uniform model]\label{rem:uniform-transform}
In addition to the moment-based inference function, one may construct weak
inference functions from the characteristic function. For $t \neq 0$, define
\[
\psi_t(x;\theta)=e^{itx}-\phi_\theta(t),
\qquad
\phi_\theta(t)=\frac{e^{it\theta}-1}{it\theta}.
\]
Then $\mathbb E_\theta[\psi_t(X;\theta)]=0$, and the corresponding Godambe
information is
\[
G_t(\theta)
=
\frac{|\partial_\theta \phi_\theta(t)|^2}
{1-|\phi_\theta(t)|^2}.
\]
This defines a family of Riemannian metrics on $(0,\infty)$: both
ingredients are nonzero for every $\theta>0$ and $t\neq 0$, since
$|\phi_\theta(t)| = 2|\sin(t\theta/2)|/(|t|\theta) < 1$, and
$\partial_\theta \phi_\theta(t) = 0$ would require
$e^{it\theta}(1-it\theta) = 1$, which is impossible because
$|e^{it\theta}(1-it\theta)| = \sqrt{1+t^2\theta^2} > 1$. The
instrument lies in $\mathcal G_{\mathrm b}$.

Moreover, as $t \to 0$, one recovers the moment-based metric
$G_t(\theta) \to 3/\theta^2$.
Thus the moment-based geometry appears as a low-frequency limit of the
transform-based geometry.
\end{remark}

\subsection{Shifted exponential model}\label{ssec:shifted-exp}

The shifted exponential model provides a second example supporting
Theorem~\ref{thm:main}. Like the uniform model, it is non-regular in the classical
likelihood sense because the support depends on the parameter. Unlike the
uniform model, it is naturally adapted to transform-based weak inference
functions, and therefore illustrates more directly the role of distributional
representations.

Consider the model
\[
X=\theta+Y,
\qquad
Y\sim \mathrm{Exp}(1),
\]
so that
$f(x;\theta)=e^{-(x-\theta)}\mathbf 1_{(x\ge \theta)}$,
$\theta\in\mathbb R$.
The support depends on $\theta$, and the formal score function is not unbiased:
for $x>\theta$, $\partial_\theta \log f(x;\theta)=1$,
so $\mathbb E_\theta[\partial_\theta \log f(X;\theta)] = 1 \neq 0$.

\paragraph{A moment-based inference function.}
Since $\mathbb E_\theta[X]=\theta+1$, a natural inference function is
$\psi(x;\theta)=x-\theta-1$.
Then $\mathbb E_\theta[\psi(X;\theta)]=0$,
$S(\theta)=1$, and $V(\theta)=1$.
Therefore $G(\theta)=1$,
and the parameter space $\mathbb R$ carries the flat Riemannian metric
$g_\theta=d\theta^2$.

\paragraph{A transform-based weak inference function.}
The characteristic function of the model is
$\phi_\theta(t) = e^{it\theta}/(1-it)$.
For $t\neq 0$, define $\psi_t(x;\theta)=e^{itx}-\phi_\theta(t)$.
Then $\mathbb E_\theta[\psi_t(X;\theta)]=0$.
Since $\partial_\theta \phi_\theta(t)=it\,\phi_\theta(t)$,
we obtain
$|\partial_\theta \phi_\theta(t)|^2 = t^2/(1+t^2)$,
$|\phi_\theta(t)|^2=1/(1+t^2)$, and
$V_t(\theta) = t^2/(1+t^2)$.
Hence $G_t(\theta) = 1$.

Thus the transform-based weak inference function induces exactly the same
Riemannian metric as the moment-based inference function. (Both
instruments are admissible: $\psi \in \mathcal G_{\mathrm{mom}}$, the
exponential having all moments, and $\psi_t \in \mathcal G_{\mathrm
b}$.)

\begin{remark}\label{rem:shifted-exp-intrinsic}
The shifted exponential model is non-regular in the classical Fisher sense, yet
it admits a smooth positive-definite metric induced by weak inference functions.
Moreover, two different inference functions---one moment-based and one
transform-based---lead to the same Godambe metric. This suggests that, while
the Godambe geometry may depend on the chosen inference function in general,
certain models may possess a more intrinsic weak geometric structure.
\end{remark}

\subsection{A fractal model: the Cantor location family}\label{ssec:cantor}

Let $C = \sum_{k\ge 1} D_k 3^{-k}$, where the $D_k$ are independent
and uniform on $\{0,2\}$: the \emph{Cantor distribution} $\mu_C$, the
natural uniform measure on the middle-thirds Cantor set $K \subset
[0,1]$ (of Hausdorff dimension $\log 2/\log 3$; see
\cite{Falconer2003}). The law $\mu_C$ is singular continuous, yet has
moments of every order, with
\[
\mathrm{E}[C] = \tfrac12,
\qquad
\mathrm{Var}(C) = \tfrac18,
\qquad
\phi_C(t) = \mathrm{E}\bigl[e^{itC}\bigr]
= e^{it/2}\prod_{k\ge 1}\cos\!\bigl(t\,3^{-k}\bigr).
\]
The centred variable $Z := C - \tfrac12$ is symmetric, with real,
even, entire characteristic function
\[
\phi_Z(t) = \prod_{k\ge 1}\cos\!\bigl(t\,3^{-k}\bigr),
\]
which is strictly positive on the window $|t| < 3\pi/2$ (each factor
$\cos(t 3^{-k})$ is positive there, the constraint being binding for
$k=1$), strictly decreasing on $(0, 3\pi/2)$ (since $(\log\phi_Z)'(t)
= -\sum_k 3^{-k}\tan(t 3^{-k}) < 0$ there), and satisfies the exact
renormalisation identity
\begin{equation}\label{eq:cantor-renorm}
\phi_Z(3t) = \cos(t)\,\phi_Z(t),
\end{equation}
the Fourier image of the triadic self-similarity $\mu_C =
\tfrac12\bigl(S_{0\#}\mu_C + S_{1\#}\mu_C\bigr)$, $S_0(x) = x/3$,
$S_1(x) = (x+2)/3$. The zero set of $\phi_Z$ is $\{3^k(2m+1)\pi/2 : k
\ge 1,\ m \in \mathbb Z\}$: the frequency-selection caveat of
Section~\ref{sec:introduction} is genuinely binding here.

Consider the location model
\[
X = \theta + C, \qquad \theta \in \Theta = \mathbb R,
\qquad T_\theta = \delta_\theta * \mu_C .
\]

\begin{proposition}[No likelihood exists]\label{prop:cantor-undominated}
\mbox{}
\begin{enumerate}
\item Each $P_\theta$ is singular continuous: no density with respect
to Lebesgue measure (nor with respect to any fixed $\sigma$-finite
measure, uniformly in $\theta$, by (2)) exists.
\item For every $\sigma$-finite measure $\nu$ on $\mathcal B(\mathbb
R)$, the set $\{\theta : P_\theta \ll \nu\}$ is Lebesgue-null. In
particular, the family $\{P_\theta : \theta \in \mathbb R\}$ is
undominated, and no choice of reference measure produces a likelihood
on any parameter set of positive Lebesgue measure.
\end{enumerate}
\end{proposition}

\begin{proof}
(1) is classical: $\mu_C(K) = 1$ with $\lambda(K) = 0$, and $\mu_C$
has no atoms. For (2), suppose $P_\theta \ll \nu$; since
$P_\theta(\theta + K) = 1$, this forces $\nu(\theta + K) > 0$. Fix a
bounded interval $[a,b]$. The map $(\theta, z) \mapsto
\mathbf 1\{z - \theta \in K\}$ is jointly measurable ($K$ is closed),
and both $\lambda$ and $\nu$ are $\sigma$-finite, so Tonelli's theorem
gives
\[
\int_a^b \nu(\theta + K)\, d\theta
= \int_{\mathbb R} \lambda\bigl([a,b] \cap (z - K)\bigr)\, \nu(dz)
= 0,
\]
because $\lambda(z - K) = \lambda(K) = 0$ for every $z$. Hence
$\nu(\theta + K) = 0$ for Lebesgue-almost every $\theta \in [a,b]$,
and since $[a,b]$ was arbitrary, $P_\theta \ll \nu$ can hold only on
a Lebesgue-null set of $\theta$.
\end{proof}

The obstruction is not curable by a cleverer choice of reference
measure: within any parameter window, densities, likelihood ratios,
and scores fail to exist simultaneously for almost every $\theta$.
The Fisher--Rao construction is void from the start---in a stronger
sense than in the support-dependent examples, where at least each
single law had a density. The weak framework, by contrast, applies
verbatim.

\paragraph{Moment instrument.}
Since $C$ has compact support, all moments exist and
$\psi(x;\theta) = x - \theta - \tfrac12 \in \mathcal G_{\mathrm{mom}}$
is a weak regular inference function, with
\[
S(\theta) = 1, \qquad V(\theta) = \mathrm{Var}(C) = \tfrac18,
\qquad G(\theta) \equiv 8 :
\]
a flat metric on $\mathbb R$, with geodesic distance
$d(\theta,\theta') = 2\sqrt 2\,|\theta - \theta'|$.

\paragraph{Transform instrument.}
For $0 < |t| < 3\pi/2$, the residual $\psi_t(x;\theta) = e^{itx} -
\phi_\theta(t) \in \mathcal G_{\mathrm b}$, with $\phi_\theta(t) =
e^{it\theta}\phi_C(t)$, gives $|\partial_\theta\phi_\theta(t)|^2 =
t^2\phi_Z(t)^2$ and $V_t = 1 - \phi_Z(t)^2 > 0$, so
\begin{equation}\label{eq:cantor-Gt}
G_t(\theta) \;=\; \frac{t^2\,\phi_Z(t)^2}{1 - \phi_Z(t)^2},
\end{equation}
a family of flat metrics indexed by the frequency. As $t \to 0$,
$\phi_Z(t) = 1 - t^2/16 + O(t^4)$ yields $G_t(\theta) \to 8$: the
moment metric is the low-frequency limit of the transform metrics,
exactly as in the uniform model
(Remark~\ref{rem:uniform-transform}).

\begin{remark}[Self-similar geometry]\label{rem:cantor-selfsimilar}
The renormalisation identity \eqref{eq:cantor-renorm} transfers to
the metric family:
\[
G_{3t}(\theta)
= \frac{9t^2 \cos^2(t)\,\phi_Z(t)^2}{1 - \cos^2(t)\,\phi_Z(t)^2}.
\]
The frequency-indexed family of Godambe metrics inherits the triadic
self-similarity of the underlying law: the geometry read at frequency
$3t$ is an explicit dressing of the geometry read at frequency $t$.
More generally, the parameter could be placed \emph{inside} the
generating iterated-function system---contraction ratio,
translations, branching probability---yielding models whose
parameters describe the geometric mechanism generating the law
rather than the law itself; weak estimating equations extend to that
setting, which we leave for separate work.
\end{remark}

\begin{remark}[Location--scale extension and automatic separation]
\label{rem:cantor-locscale}
The two-parameter family $X = \mu + \sigma Z$ is a symmetric
location--scale model with real, entire, non-constant $\phi_Z$, so
Proposition~\ref{prop:automatic-orthogonality} applies verbatim: for
frequencies with $|c\sigma|, |d\sigma| \in (0, 3\pi/2)$ one has
$\phi_Z(c\sigma) > 0$ and $\phi_Z'(d\sigma) < 0$, and the sine/cosine
probes produce a block-diagonal Godambe metric---exact location--scale
inferential separation in a family that possesses no likelihood at
all. (Positivity of the variability uses $|\phi_Z(t)| < 1$ for $t
\neq 0$, which holds because $\mu_C$ is non-atomic, hence non-lattice.)
\end{remark}

\begin{remark}[Instrument-relative information]
\label{rem:cantor-instrument}
From a single observation, $\theta$ is confined to the Lebesgue-null
set $x - K$: the identification structure of the model is far
stronger than any single $\sqrt n$-type quantification can express.
The Godambe metric measures the information extracted by the
\emph{chosen instrument}, not an intrinsic bound; in an undominated
model there is no Fisher information to saturate, and different
instruments read genuinely different geometries; indeed the family
of Godambe metrics has no maximal element
(Remark~\ref{rem:no-max}).
\end{remark}

\subsection{Stratified finite mixtures and the failure of the score route}\label{ssec:mixture}

A different type of obstruction is illustrated by the stratified finite
mixed models of \cite{Labouriau2022}. Let $\mathcal P_0 = \{P_\vartheta
: \vartheta \in \Theta_0\}$ be a regular dominated family (the
\emph{basic model}), and suppose the observations $X_1,\dots,X_I$ are
stratified into $K$ groups: each $X_i$ is drawn from one, and only
one, of $K$ distinct component laws $P_{\vartheta_1}, \dots,
P_{\vartheta_K}$. The group memberships are not fully observed;
instead, a known \emph{mixing matrix} $\Pi = [\pi_{ik}]$, with
$\pi_{ik} \in [0,1]$ and $\sum_k \pi_{ik} = 1$, records the
probability that observation $i$ belongs to group $k$, and the
\emph{mixed model} attributes to $X_i$ the density $\sum_k \pi_{ik}\,
p(\cdot\,;\vartheta_k)$. If all $\pi_{ik} \in \{0,1\}$, the group
structure is fully observed (a classification factor) and the model
\emph{contains no mixture}; if some row has entries in $(0,1)$, the
model \emph{contains a genuine mixture}.

The key result of \cite{Labouriau2022} (Proposition~1 there) is that
the score function of the mixed model is an unbiased inference
function \emph{if, and only if, the model contains no mixture}; in the
presence of a genuine mixture the score is biased---the expectation of
the score for $\vartheta_k$, evaluated under the component law
actually generating the observation, is strictly negative---and,
under further mild regularity, the maximum likelihood estimator is
inconsistent. The model is dominated, with a smooth, everywhere
positive likelihood; the failure of the Fisher route is purely
inferential. From the geometric point of view, this shows that one
should not identify the existence of a likelihood or a formal score
with the existence of a meaningful information metric: the relevant
structure is a regular inference function with well-defined
sensitivity and variability.

\paragraph{Stratification-agnostic instruments.}
The distributional framework suggests where to look. Because the
memberships are unknown, an instrument for $\theta =
(\vartheta_1,\dots,\vartheta_K)$ should be \emph{simultaneously
weakly unbiased under every component}:
\begin{equation}\label{eq:simultaneous-unbiased}
\langle T_{\vartheta_k},\, \psi_j(\cdot\,;\theta)\rangle = 0,
\qquad k = 1,\dots,K,\ j = 1,\dots,q .
\end{equation}
An inference function satisfying
\eqref{eq:simultaneous-unbiased} is unbiased under \emph{every} law in
the convex hull $\mathrm{co}\{P_{\vartheta_1},\dots,P_{\vartheta_K}\}$,
hence under every realised stratification and every admissible mixing
matrix $\Pi$: it is \emph{stratification-agnostic}, and the bias
mechanism of \cite{Labouriau2022} cannot touch it. (A dimension count
shows such instruments exist in abundance: within the polynomials of
degree at most $2K-1$, the simultaneous unbiasedness constraints
\eqref{eq:simultaneous-unbiased} cut down $2K$ coefficients by $K$
conditions, leaving a $K$-dimensional family---exactly enough to
estimate $K$ group parameters.)

\begin{proposition}[Godambe metric for the stratified Gaussian mixture]\label{prop:mixture}
Let $K=2$ with Gaussian components $\mathcal N(\mu_1,\sigma^2)$,
$\mathcal N(\mu_2,\sigma^2)$, $\sigma^2$ known, and parameter space
$\Theta = \{(\mu_1,\mu_2) : \mu_1 < \mu_2\}$. Define
\[
W(x;\theta) =
\begin{pmatrix}
\psi_a(x;\theta)\\[2pt]
\psi_b(x;\theta)
\end{pmatrix},
\]
where
\[
\begin{aligned}
\psi_a(x;\theta) &= (x-\mu_1)(x-\mu_2) - \sigma^2,\\
\psi_b(x;\theta) &= x^3 - (\mu_1^2+\mu_1\mu_2+\mu_2^2+3\sigma^2)\,x
                    + \mu_1\mu_2(\mu_1+\mu_2).
\end{aligned}
\]
Then:
\begin{enumerate}
\item $W$ is simultaneously weakly unbiased,
$\mathrm{E}_{\mathcal N(\mu_k,\sigma^2)}[W(X;\theta)] = 0$ for
$k=1,2$, hence unbiased under every mixing matrix $\Pi$; its
components lie in $\mathcal G_{\mathrm{mom}}$
(Lemma~\ref{lem:extension}).
\item Let $\bar w_k := I^{-1}\sum_{i=1}^I \pi_{ik}$ denote the average
design weight of group $k$, and let $S(\theta)$, $V(\theta)$ be the
averaged (per-observation) sensitivity and variability of $W$ under
the mixed model. Because expectations under the mixed law are linear
in the weights, $S$ and $V$ depend on $\Pi$ only through
$(\bar w_1, \bar w_2)$, and
\[
S(\theta) = -\,\delta
\begin{pmatrix}
\bar w_1 & -\bar w_2\\
\bar w_1(2\mu_1+\mu_2) & -\bar w_2(\mu_1+2\mu_2)
\end{pmatrix},
\qquad
\delta := \mu_2 - \mu_1 > 0,
\]
so that
$|\det S(\theta)| = \bar w_1 \bar w_2\, \delta^3$.
\item If both groups are effectively present, $\bar w_1 \bar w_2 > 0$,
then $S(\theta)$ is nonsingular, $V(\theta)$ is positive definite, and
the Godambe information
$G(\theta) = S(\theta)^\top V(\theta)^{-1} S(\theta)$
defines a smooth Riemannian metric on $\Theta$.
\end{enumerate}
\end{proposition}

\begin{proof}
(1) For $Y \sim \mathcal N(\mu_k, \sigma^2)$ write $X = Y$ and expand
about $\mu_k$. For $\psi_a$:
$\mathrm{E}[(X-\mu_1)(X-\mu_2)] = \sigma^2 +
(\mu_k-\mu_1)(\mu_k-\mu_2) = \sigma^2$ for $k \in \{1,2\}$, since one
factor vanishes. For $\psi_b$, using
$\mathrm{E}[X] = \mu_k$, $\mathrm{E}[X^3] = \mu_k^3 + 3\mu_k\sigma^2$:
\[
\mathrm{E}[\psi_b(X;\theta)]
= \mu_k^3 + 3\mu_k\sigma^2
- (\mu_1^2+\mu_1\mu_2+\mu_2^2+3\sigma^2)\mu_k
+ \mu_1\mu_2(\mu_1+\mu_2),
\]
which vanishes for $\mu_k \in \{\mu_1, \mu_2\}$ by direct expansion.
Unbiasedness under any law in
$\mathrm{co}\{\mathcal N(\mu_1,\sigma^2), \mathcal N(\mu_2,\sigma^2)\}$
follows by linearity of the pairing, and hence under any $\Pi$.

(2) The per-observation sensitivity under observation $i$ is linear in
the row $\pi_{i\cdot}$, so the average over $i$ equals the sensitivity
evaluated at the averaged weights $(\bar w_1, \bar w_2)$; the same
holds for the variability. The entries follow from
$\partial_{\mu_1}\psi_a = -(x - \mu_2)$,
$\partial_{\mu_2}\psi_a = -(x - \mu_1)$,
$\partial_{\mu_1}\psi_b = -(2\mu_1+\mu_2)x + \mu_2(2\mu_1+\mu_2)$,
$\partial_{\mu_2}\psi_b = -(\mu_1+2\mu_2)x + \mu_1(\mu_1+2\mu_2)$,
whose component-law expectations are
$\mathrm{E}_k[\partial_{\mu_1}\psi_a] = -(\mu_k-\mu_2)$,
$\mathrm{E}_k[\partial_{\mu_2}\psi_a] = -(\mu_k-\mu_1)$, and
$\mathrm{E}_1[\partial_{\mu_1}\psi_b] = \delta(2\mu_1+\mu_2)$,
$\mathrm{E}_2[\partial_{\mu_1}\psi_b] = 0$,
$\mathrm{E}_1[\partial_{\mu_2}\psi_b] = 0$,
$\mathrm{E}_2[\partial_{\mu_2}\psi_b] = -\delta(\mu_1+2\mu_2)$.
Averaging with weights $(\bar w_1,\bar w_2)$ and applying
$S = -\mathrm{E}[\partial_\theta W]$ gives the stated matrix, and
\[
\det S(\theta)
= \bar w_1 \bar w_2\,\delta^2\,
\bigl[(\mu_1+2\mu_2)-(2\mu_1+\mu_2)\bigr]\cdot(-1)
= -\,\bar w_1 \bar w_2\, \delta^3 .
\]

(3) Nonsingularity of $S$ is immediate from (2). For $V$: for any $a =
(a_1,a_2)^\top \neq 0$, $a^\top V(\theta) a$ is the variance of the
polynomial $a_1\psi_a + a_2\psi_b$ (of exact degree $2$ or $3$) under
the mixed law, whose density is strictly positive on $\mathbb R$; a
non-constant polynomial cannot be constant almost surely under such a
law, and $a_1\psi_a + a_2\psi_b$ is non-constant whenever $a \neq 0$.
Hence $V(\theta) \succ 0$, and
Proposition~\ref{prop:godambe-metric} applies.
\end{proof}

\begin{remark}\label{rem:mixture-complement}
This example is complementary to the support-dependent models
considered earlier. The mixed model is dominated and admits a smooth
positive likelihood, yet the score route fails in the precise sense of
\cite{Labouriau2022}: with a genuine mixture the score is biased and
the MLE inconsistent. The stratification-agnostic instrument $W$ is
untouched by this mechanism, and induces a smooth Godambe metric on
$\Theta$ for \emph{every} admissible mixing matrix. The geometry
degenerates exactly when a group leaves the effective sample
($\bar w_1 \bar w_2 \to 0$) or the components merge ($\delta \to 0$),
which is the correct boundary behaviour: in either limit the two-group
parametrisation genuinely loses identifiability.
\end{remark}

\begin{remark}[Mixtures of moment-free distributions]\label{rem:cauchy-mixture}
The stratified Gaussian example isolates the phenomenon of score bias by
working with component distributions that have moments of all orders.
The construction extends directly to components lacking ordinary
moments. For instance, consider the two-component Cauchy location
mixture with known mixing weight $w \in (0,1)$,
\[
f_\theta(x)
= w \cdot \frac{1}{\pi(1 + (x - \mu_1)^2)}
+ (1-w) \cdot \frac{1}{\pi(1 + (x - \mu_2)^2)},
\qquad \theta = (\mu_1, \mu_2).
\]
No ordinary moments exist, yet the \emph{distributional} moments are
well defined (see Section~\ref{ssec:distr-moments} and \cite{C}), and
one may construct a weak inference function from the characteristic
function. Since the characteristic function of the Cauchy location
model is $\phi_{\mu}(t) = e^{it\mu - |t|}$, the mixture characteristic
function is
$\phi_\theta(t) = e^{-|t|}(w\, e^{it\mu_1} + (1-w)\, e^{it\mu_2})$,
and transform-based weak inference functions of the form
$\psi_t(x;\theta) = e^{itx} - \phi_\theta(t)$, which lie in
$\mathcal G_{\mathrm b}$, may be used to construct a Godambe metric.
This illustrates that the Godambe--Riemannian framework simultaneously
handles two distinct obstructions: score bias (as in the stratified
Gaussian mixture) and the absence of ordinary moments.
\end{remark}

\begin{remark}[Geodesic distances: the geometry is usable]\label{rem:geodesics}
In the one-parameter examples the metrics integrate to explicit
Riemannian distances: $d(\theta,\theta') =
\sqrt3\,\bigl|\log(\theta'/\theta)\bigr|$ for the uniform scale model
(Remark~\ref{rem:uniform-log}), $d(\theta,\theta') =
|\theta-\theta'|$ for the shifted exponential, and $d(\theta,\theta')
= 2\sqrt2\,|\theta-\theta'|$ for the Cantor location model (moment
instruments). For the stratified mixture the distance is the
Riemannian distance of the two-parameter metric $G(\theta)$ on
$\{\mu_1 < \mu_2\}$, available numerically. The Godambe geometries
constructed here are therefore immediately usable---for
reparametrisation-invariant step sizes, preconditioning, and
model-distance diagnostics---while curvature questions, which require
multiparameter families, are deferred to
Section~\ref{sec:discussion}.
\end{remark}

\section{A lattice SPDE example: stability of the statistical model}\label{sec:spde}

The three examples of Section~\ref{sec:examples} were chosen to
isolate three qualitatively distinct obstructions to the
Fisher--Rao construction. In this section we turn to a structurally
richer example, coming from the theory of stochastic partial
differential equations, in which the Godambe--Riemannian framework
not only provides a Riemannian metric where none was available, but
also admits a stability statement that ties the geometry of the
statistical model to the spectral properties of the underlying
dynamical system. The example is finite-dimensional throughout---it
is posed on a finite lattice and all the computations reduce to
one-dimensional $\alpha$-stable arithmetic in the eigenbasis of the
discrete Laplacian---so no infinite-dimensional analysis is required.

\subsection{The lattice stochastic heat equation}\label{ssec:spde-setup}

Fix $N \ge 3$. Consider the discrete stochastic heat equation on the
lattice $\{1,\dots,N-1\}$ with Dirichlet boundary conditions
$u_0(t) = u_N(t) = 0$:
\begin{equation}\label{eq:spde}
du_i(t) \;=\; -\theta \,(A u(t))_i \, dt \;+\; \sigma \, dL_i^{(\alpha)}(t),
\qquad i = 1, \dots, N-1,
\end{equation}
where $A = -\Delta_N$ is the $(N-1)\times (N-1)$ discrete negative
Laplacian,
\[
A \;=\; \mathrm{tridiag}(-1,2,-1),
\]
$\theta > 0$ is the diffusion parameter (the single parameter to be
estimated), $\sigma > 0$ is a known scale, and
$L^{(\alpha)}_1,\dots,L^{(\alpha)}_{N-1}$ are independent symmetric
$\alpha$-stable L\'evy processes with stability index
$\alpha \in (0,2)$ and unit scale. The initial condition is $u(0)=0$.
No moment assumptions enter anywhere below: for $\alpha < 2$ the noise
has no variance, and for $\alpha \le 1$ not even a mean.

The operator $A$ is symmetric positive definite, with eigenvalues and
normalised eigenvectors
\begin{equation}\label{eq:laplacian-spectrum}
\lambda_k \;=\; 4\sin^2\!\left(\frac{k\pi}{2N}\right),
\qquad
v_k(i) \;=\; \sqrt{\tfrac{2}{N}}\,\sin\!\left(\frac{k\pi i}{N}\right),
\qquad k,i = 1,\dots,N-1.
\end{equation}
In particular $\lambda_1(N) = 4\sin^2(\pi/(2N))$ is the spectral gap
of $A$, which controls the rate of decay of the slowest eigenmode of
the deterministic heat flow $\dot u = -\theta A u$.

Because $L^{(\alpha)}$ has no variance for $\alpha < 2$, the state
$u(t)$ has no covariance matrix and the classical Fisher--Rao
machinery is unavailable: the model has no density in closed form,
and the usual second-moment Fisher information is simply undefined.
The distributional framework of Section~\ref{sec:preliminaries}
applies directly, however, because the characteristic function of
$u(t)$ is available in closed form.

\subsection{Closed-form characteristic function}\label{ssec:spde-charfun}

Since the drift in \eqref{eq:spde} is linear and $A$ is symmetric,
Duhamel's formula yields
\[
u(t) \;=\; \sigma \int_0^t e^{-\theta (t-s) A} \, dL^{(\alpha)}(s).
\]
Using the independence and scaling properties of symmetric
$\alpha$-stable L\'evy noise, the characteristic function of $u(t)$
at a test vector $\xi \in \mathbb R^{N-1}$ is
\[
\phi_\theta(\xi; t)
\;=\; \mathbb E_\theta\!\left[e^{i \langle \xi, u(t)\rangle}\right]
\;=\; \exp\!\left( -\sigma^\alpha \int_0^t
\big\| e^{-\theta(t-s) A} \xi \big\|_\alpha^\alpha \, ds \right),
\]
where $\|x\|_\alpha^\alpha = \sum_{i=1}^{N-1} |x_i|^\alpha$. Choosing
$\xi = c v_k$ along a single eigenvector of $A$, one obtains the
clean formula
\begin{equation}\label{eq:phi-mode}
\phi_\theta(c v_k; t)
\;=\; \exp\!\left( - |c|^\alpha \beta_k(t,\theta)\right),
\end{equation}
where
\begin{equation}\label{eq:beta}
\beta_k(t,\theta)
\;=\;
\sigma^\alpha m_k \,
\frac{1 - e^{-\alpha\theta\lambda_k t}}{\alpha\theta\lambda_k},
\qquad
m_k \;:=\; \|v_k\|_\alpha^\alpha
      \;=\; \left(\tfrac{2}{N}\right)^{\alpha/2}
            \sum_{i=1}^{N-1}
            \left|\sin\!\left(\tfrac{k\pi i}{N}\right)\right|^\alpha.
\end{equation}
In particular, the projection $U_k(t) := \langle v_k, u(t)\rangle$ is
univariate symmetric $\alpha$-stable with scale
$s_k(t,\theta) = \beta_k(t,\theta)^{1/\alpha}$, and its stationary
scale is
\[
\beta_k(\infty,\theta) \;=\;
\frac{\sigma^\alpha m_k}{\alpha\theta\lambda_k},
\]
from which the exponential approach to stationarity
\begin{equation}\label{eq:beta-rate}
\beta_k(t,\theta) - \beta_k(\infty,\theta)
\;=\; -\,\beta_k(\infty,\theta)\, e^{-\alpha\theta\lambda_k t}
\end{equation}
is immediate.

\subsection{Weak Godambe information from eigenmode probes}\label{ssec:spde-godambe}

Since $U_k(t)$ is symmetric and has no finite variance, the natural
probe is a bounded cosine inference function. Fix $k$ and $c \neq 0$
and define
\begin{equation}\label{eq:psi-spde}
\psi_{k,c}(u;\theta)
\;=\; \cos\!\big( c\, U_k(t) \big)
      \;-\; f_k(c, t, \theta),
\qquad
f_k(c, t, \theta) \;:=\; e^{-|c|^\alpha \beta_k(t,\theta)}.
\end{equation}
By construction $\psi_{k,c}$ is bounded by $2$, lies in the admissible
class $\mathcal G_{\mathrm b}$ of Lemma~\ref{lem:extension} (it is a
bounded continuous function of the field $u$), and is weakly unbiased:
it is a weak regular inference function in the sense of
Definition~\ref{def:rif}. Its
weak sensitivity and weak variability are closed-form:
\begin{align}
S_{k,c}(t,\theta)
  &= -\,\mathrm{E}_\theta[\partial_\theta \psi_{k,c}]
   = -\,|c|^\alpha \, \partial_\theta \beta_k(t,\theta)\,
       f_k(c,t,\theta),
     \label{eq:spde-sensitivity}\\[2pt]
V_{k,c}(t,\theta)
  &= \mathrm{E}_\theta[\psi_{k,c}^2]
   = \tfrac{1}{2}\big(1 + f_k(2c,t,\theta)\big)
     \;-\; f_k(c,t,\theta)^2,
     \label{eq:spde-variability}
\end{align}
with
\begin{equation}\label{eq:dbeta-dtheta}
\partial_\theta \beta_k(t,\theta)
\;=\; \frac{\sigma^\alpha m_k \, t\, e^{-\alpha\theta\lambda_k t}
      \;-\;\beta_k(t,\theta)}
      {\theta}.
\end{equation}
Positivity of the variability holds for every $t > 0$ and all
$\alpha \in (0,2)$: writing $x := |c|^\alpha \beta_k(t,\theta) > 0$,
\[
V_{k,c} \;=\; \tfrac12\bigl(1 + e^{-2^\alpha x}\bigr) - e^{-2x}
\;>\; \tfrac12\bigl(1 + e^{-4 x}\bigr) - e^{-2x}
\;=\; \tfrac12\bigl(1 - e^{-2x}\bigr)^2 \;\ge\; 0,
\]
since $2^\alpha < 4$. The associated weak Godambe information is
\begin{equation}\label{eq:Gspde}
G_k(t,\theta,c) \;=\; \frac{S_{k,c}(t,\theta)^2}{V_{k,c}(t,\theta)},
\end{equation}
and the stationary limit $G_k(\infty,\theta,c)$ is obtained by
substituting $\beta_k(\infty,\theta) = \sigma^\alpha m_k /
(\alpha\theta\lambda_k)$ and
$\partial_\theta\beta_k(\infty,\theta) = -\beta_k(\infty,\theta)/\theta$
into \eqref{eq:spde-sensitivity}--\eqref{eq:Gspde}. In particular,
$G_k(\infty,\theta,c)$ is nontrivial for any $c\neq 0$, and the choice
$c_\star := \beta_k(\infty,\theta)^{-1/\alpha}$ gives the universal
normalisation $|c_\star|^\alpha\beta_k(\infty,\theta) = 1$.

Since the parameter is scalar, a single eigenmode probe already
realises the Godambe--Riemannian structure guaranteed by
Theorem~\ref{thm:main}, in closed form, for a model with no variance
and no density. Combining several probes raises a point that deserves
emphasis. For $\alpha = 2$ the projections $U_k(t) = \langle v_k,
u(t)\rangle$ onto distinct eigenmodes are independent (orthogonal
rotations preserve the independence of Gaussian coordinates), but for
$\alpha < 2$ they are \emph{dependent}: joint $\alpha$-stable vectors
do not decouple under orthogonal transformations, and the
cross-moments of eigenmode probes do not vanish in general. A sum
$\sum_{j} G_{k_j}(t,\theta,c_j)$ is therefore \emph{not} the Godambe
information of the stacked inference function; it is the metric of
the diagonally weighted quadratic discrepancy with $W =
\mathrm{diag}(1/V_{k_j,c_j})$, legitimate by
Proposition~\ref{prop:stein-godambe} but suboptimal. The exact
multi-probe Godambe information is nevertheless available in closed
form, because the joint characteristic function of any finite family
of eigenmode projections is explicit; we carry this out for two
probes.

\subsection{The exact two-probe Godambe information}\label{ssec:spde-two-probe}

Fix two distinct modes $k \neq l$, frequencies $c, d \neq 0$, and
stack $\psi = (\psi_{k,c}, \psi_{l,d})^\top$. The sensitivity vector
$S = (S_{k,c}, S_{l,d})^\top$ has the marginal entries
\eqref{eq:spde-sensitivity}, and the diagonal entries of $V$ are
\eqref{eq:spde-variability}. For the off-diagonal entry, the
product-to-sum identity $\cos A \cos B = \tfrac12[\cos(A+B) +
\cos(A-B)]$ and the closed-form joint characteristic function give
\begin{equation}\label{eq:spde-crosscov}
V_{12}(t,\theta)
= \mathrm{E}_\theta[\psi_{k,c}\,\psi_{l,d}]
= \tfrac12\Bigl( e^{-\beta^{+}_{k,l}(t,\theta)}
               + e^{-\beta^{-}_{k,l}(t,\theta)} \Bigr)
- e^{-|c|^\alpha \beta_k(t,\theta)}\, e^{-|d|^\alpha \beta_l(t,\theta)},
\end{equation}
where, substituting $u = t-s$,
\begin{equation}\label{eq:beta-pm}
\beta^{\pm}_{k,l}(t,\theta)
\;:=\; \sigma^\alpha \int_0^t \sum_{i=1}^{N-1}
\bigl| c\, e^{-\theta\lambda_k u} v_k(i)
  \,\pm\, d\, e^{-\theta\lambda_l u} v_l(i) \bigr|^\alpha \, du ,
\end{equation}
so that $\mathrm{E}_\theta[\cos(cU_k \pm dU_l)] =
\phi_\theta(cv_k \pm dv_l;t) = e^{-\beta^{\pm}_{k,l}(t,\theta)}$.
For $\alpha = 2$,
orthonormality of $(v_k)$ collapses \eqref{eq:beta-pm} to
$\beta^{\pm}_{k,l} = c^2\beta_k + d^2\beta_l$ and
\eqref{eq:spde-crosscov} vanishes identically---the Gaussian
decoupling. For $\alpha < 2$ the $\alpha$-norm is not quadratic, the
cross terms in \eqref{eq:beta-pm} do not cancel, and $V_{12} \neq 0$
in general.

The exact two-probe weak Godambe information (with $p = 1$, $q = 2$)
is then the scalar
\begin{equation}\label{eq:G-two-probe}
G_{k,l}(t,\theta;c,d)
= S^\top V^{-1} S
= \frac{S_{k,c}^2 V_{22} - 2 S_{k,c} S_{l,d} V_{12}
        + S_{l,d}^2 V_{11}}
       {V_{11} V_{22} - V_{12}^2} ,
\end{equation}
with all ingredients given by
\eqref{eq:spde-sensitivity}--\eqref{eq:dbeta-dtheta} and
\eqref{eq:spde-crosscov}--\eqref{eq:beta-pm}: everything reduces to
finite sums and one-dimensional integrals of elementary functions.
The matrix $V$ is nonsingular for $t>0$: strict Cauchy--Schwarz holds
because the joint law of $(U_k, U_l)$ has full support, so the two
probe residuals are not almost-surely proportional. By
Proposition~\ref{prop:godambe-metric} (with full column rank $q=2 >
p=1$), \eqref{eq:G-two-probe} defines a Riemannian metric on
$\Theta=(0,\infty)$, and $G_{k,l} \ge \max\{G_k, G_l\}$, with strict
improvement whenever $V_{12} \neq 0$.

The stabilisation analysis extends verbatim: the tail bound
\[
\bigl|\beta^{\pm}_{k,l}(t,\theta) - \beta^{\pm}_{k,l}(\infty,\theta)\bigr|
\;\le\; \sigma^\alpha (N-1)
\bigl( |c|\,\|v_k\|_\infty + |d|\,\|v_l\|_\infty \bigr)^{\alpha}
\frac{e^{-\alpha\theta\lambda_{k\wedge l} t}}{\alpha\theta\lambda_{k\wedge l}},
\]
where $\lambda_{k \wedge l} := \min\{\lambda_k, \lambda_l\}$, shows
that all entries of $S$ and $V$ converge to their stationary values at
rate at least $\alpha\theta\lambda_{k\wedge l}$ (with the polynomial
prefactor of \eqref{eq:dbeta-dtheta} for the sensitivities), whence,
by the same Taylor argument as in
Proposition~\ref{prop:spde-stability},
\[
\bigl| G_{k,l}(t,\theta;c,d) - G_{k,l}(\infty,\theta;c,d) \bigr|
\;\le\; C\,(1+t)\, e^{-\alpha\theta\lambda_{k\wedge l}\, t}:
\]
the two-probe geometry stabilises at the rate of the \emph{slower} of
the two modes, exactly as the dynamical picture dictates.

\subsection{The spectral gap drives geometric stability}\label{ssec:spde-stability}

We can now state the central observation of this section: the rate at
which the Godambe geometry of the statistical model stabilises as the
observation time $t$ grows is controlled by the spectrum of $-\Delta_N$
itself, i.e., by the same quantity that controls the dynamical
stability of the deterministic heat flow.

\begin{proposition}[Spectral-gap stability of the Godambe metric]
\label{prop:spde-stability}
Fix $N\ge 3$, $\theta > 0$, $\sigma > 0$, $\alpha \in (1,2)$,
$k \in \{1,\dots,N-1\}$ and $c \in \mathbb R \setminus \{0\}$. The
weak Godambe information $G_k(t,\theta,c)$ defined by
\eqref{eq:Gspde} converges to its stationary value
$G_k(\infty,\theta,c)$ as $t \to \infty$, and there exists a
constant $C = C(\theta,\sigma,\alpha,k,c,N) > 0$ such that
\begin{equation}\label{eq:spde-stability}
\big| G_k(t,\theta,c) - G_k(\infty,\theta,c)\big|
\;\le\; C\,(1+t)\,e^{-\alpha\theta\lambda_k t},
\end{equation}
where $\lambda_k = 4\sin^2(k\pi/(2N))$. In particular,
\[
\limsup_{t\to\infty}\;\frac{1}{t}\,
\log\!\big|G_k(t,\theta,c) - G_k(\infty,\theta,c)\big|
\;\le\; -\,\alpha\theta\lambda_k,
\]
and the slowest stabilisation rate, obtained for $k=1$, equals
$\alpha\theta\lambda_1(N) = 4\alpha\theta\sin^2(\pi/(2N))$.
\end{proposition}

\begin{proof}
From \eqref{eq:beta-rate}, $\beta_k(t,\theta) - \beta_k(\infty,\theta)
= -\beta_k(\infty,\theta)e^{-\alpha\theta\lambda_k t}$, which is
exponentially small with rate $\alpha\theta\lambda_k$. Substituting
this into \eqref{eq:dbeta-dtheta} yields
\[
\partial_\theta \beta_k(t,\theta)
\;-\; \big(-\beta_k(\infty,\theta)/\theta\big)
\;=\; \frac{\sigma^\alpha m_k\, t + \beta_k(\infty,\theta)}{\theta}\,
      e^{-\alpha\theta\lambda_k t},
\]
which has a polynomial prefactor $(1+t)$ but decays with the same
exponential rate. The functions $f_k(c,t,\theta)$ and
$f_k(2c,t,\theta)$, and hence the sensitivity
\eqref{eq:spde-sensitivity} and variability
\eqref{eq:spde-variability}, are smooth functions of
$\beta_k(t,\theta)$ and $\partial_\theta \beta_k(t,\theta)$; a
first-order Taylor expansion around the stationary values gives the
bound \eqref{eq:spde-stability} with the polynomial prefactor
inherited from \eqref{eq:dbeta-dtheta}. Taking logarithms and
dividing by $t$ yields the $\limsup$ statement.
\end{proof}

\begin{remark}[Collapse under unstable dynamics]\label{rem:spde-collapse}
The stability statement has a two-sided reading. For $\theta > 0$ the
geometry stabilises at rate $\alpha\theta\lambda_k$. If instead
$\theta < 0$ (unstable heat flow), then $\beta_k(t,\theta) =
\sigma^\alpha m_k\,(e^{\alpha|\theta|\lambda_k t} -
1)/(\alpha|\theta|\lambda_k) \to \infty$, so $f_k(c,t,\theta) =
e^{-|c|^\alpha\beta_k} \to 0$ doubly exponentially. The decay of
$f_k$ dominates the polynomial and exponential growth of
$\partial_\theta\beta_k$ in \eqref{eq:spde-sensitivity}, whence
$S_{k,c} \to 0$, while $V_{k,c} \to \tfrac12$; therefore
$G_k(t,\theta,c) \to 0$. At long horizons a fixed bounded probe
extracts no information about an unstable parameter: the state
diverges and the cosine probe saturates. Dynamical instability thus
appears geometrically as the \emph{collapse}---not the blow-up---of
the Godambe metric.
\end{remark}

\begin{remark}[The stationary geometry is scale geometry]\label{rem:spde-adaptive}
At stationarity, $\theta$ enters each probe only through the scale
$\beta_k(\infty,\theta) = \sigma^\alpha m_k/(\alpha\theta\lambda_k)
\propto 1/\theta$. Writing $x := |c|^\alpha\beta_k(\infty,\theta)$,
the stationary metric takes the separated form
\[
G_k(\infty,\theta,c) \;=\; \frac{h_\alpha(x)}{\theta^2},
\qquad
h_\alpha(x)
:= \frac{x^2 e^{-2x}}
        {\tfrac12\bigl(1+e^{-2^\alpha x}\bigr) - e^{-2x}} ,
\]
since $S_{k,c}(\infty,\theta) = (x/\theta)e^{-x}$ by
$\partial_\theta\beta_k(\infty,\theta) =
-\beta_k(\infty,\theta)/\theta$. The profile $h_\alpha$ vanishes at
$0$ and at $\infty$ and attains an interior maximum $h_\alpha^* =
h_\alpha(x^*_\alpha)$; consequently
\[
\sup_{c \neq 0}\, G_k(\infty,\theta,c)
\;=\; \frac{h_\alpha^*}{\theta^2}
\qquad\text{for every mode } k .
\]
The frequency-optimised stationary Godambe metric is thus exactly the
scale-invariant metric $\mathrm{const}\cdot d\theta^2/\theta^2$---flat
in $\eta = \log\theta$, with geodesic distance proportional to
$|\log(\theta'/\theta)|$---the same geometry carried by the uniform
scale model (Remark~\ref{rem:uniform-log}). This is as it must be:
at stationarity $\theta$ is a pure scale parameter of each eigenmode
law, and the optimised weak geometry detects precisely that. (The
optimisation is over the fixed-frequency family; a probe with
$\theta$-dependent frequency is a different instrument, whose
sensitivity acquires an extra term.)
\end{remark}

\subsection{Reading and numerical verification}\label{ssec:spde-reading}

Proposition~\ref{prop:spde-stability} is best read as a quantitative
bridge between the \emph{dynamical} and the \emph{geometric}
stability of the model. On the dynamical side, the deterministic
flow $\dot u = -\theta A u$ decays in each eigenmode at rate
$\theta\lambda_k$; the slowest mode decays at rate
$\theta\lambda_1(N)$, which is the classical spectral gap of $-\Delta_N$.
On the geometric side, the weak Godambe information $G_k(t,\theta,c)$
stabilises at rate $\alpha\theta\lambda_k$. The two rates differ by
the factor $\alpha$, reflecting the fact that the characteristic
function of $\alpha$-stable noise carries an $\alpha$-th power of the
scale; as $\alpha \to 2$ (the Gaussian limit) the geometric rate
becomes $2\theta\lambda_k$, matching the rate at which the
second-moment covariance of a Gaussian OU process converges to its
stationary value. In either case, the Godambe geometry of the
statistical model stabilises at precisely the rate dictated by the
dynamical stability of the underlying lattice heat equation, the two
rates agreeing up to the explicit noise-index factor $\alpha$.

Seen through the lens of the roles of the metric discussed in
Section~\ref{sec:canonicity}, the example simultaneously exercises
the inferential role (the Godambe information quantifies how fast one
can learn $\theta$ from data at time $t$) and the dynamical/geometric
role (the same quantity encodes the stability of the underlying
system). That these two roles are linked by an explicit rate is not
an accident of the Gaussian theory---it holds in the $\alpha$-stable
case without any variance in sight---and it points to a general link between the dynamical (Lyapunov)
stability of the underlying system and the geometric stability of the
associated statistical model.

We have verified the main formulas
\eqref{eq:phi-mode}--\eqref{eq:beta-rate} numerically against a direct
Monte Carlo simulation of \eqref{eq:spde} (Euler--Maruyama with
Chambers--Mallows--Stuck $\alpha$-stable increments): the empirical
characteristic function of $u(t)$ along $v_k$ agrees with the
closed form to within the Monte Carlo error. Similarly,
Proposition~\ref{prop:spde-stability} has been verified numerically
on the closed-form expression for $G_k(t,\theta,c)$, with the
empirical decay rate of $|G_k(t) - G_k(\infty)|$ converging to
$\alpha\theta\lambda_k$ as $t$ is pushed into the asymptotic
regime (the finite-$t$ deficit is entirely accounted for by the
$(1+t)$ prefactor in the bound).

\section{Stein discrepancies and geometry}\label{sec:stein}

An alternative route to the construction of a Riemannian structure on a
statistical model is provided by Stein discrepancies \cite{D}. This approach is
conceptually distinct from both likelihood-based and
inference-function-based constructions. While the Fisher information arises
from the local behaviour of the likelihood, and the Godambe information
arises from inference functions, Stein discrepancies are based on operator
identities that characterise the model.

We use the notion of weak Stein operator introduced in
Section~\ref{ssec:weak-stein}. Given a parametric family
$\{P_\theta:\theta\in\Theta\subseteq\mathbb R^p\}$ and a weak Stein
operator $\mathcal A_\theta$, a Stein discrepancy between $Q$ and
$P_\theta$ is defined as
\[
D_{\mathrm S}(Q,P_\theta)
=
\sup_{g\in\mathcal G}
\left|
\mathbb E_Q[\mathcal A_\theta g(X)]
\right|
\]
or through a quadratic form involving a finite or infinite collection of test
functions. When $\mathcal G$ is a class of Schwartz functions, this
construction does not require the existence of a density; it depends only on
the distributional representation of the model.

\subsection{The Stein--Godambe equivalence}\label{ssec:stein-godambe}

To investigate whether Stein discrepancies induce a Riemannian structure, consider
their local behaviour along the model. For $h\in\mathbb R^p$ small, if the
discrepancy is sufficiently smooth, one expects an expansion
\[
D_{\mathrm S}(P_{\theta+h},P_\theta)^2
=
h^\top G_{\mathrm S}(\theta) h
+
o(\|h\|^2),
\qquad h\to 0,
\]
where $G_{\mathrm S}(\theta)$ is a symmetric matrix. If $G_{\mathrm S}(\theta)$
is positive definite and smooth, it defines a Riemannian metric, which we call
a Stein metric.

\begin{proposition}[Stein--Godambe equivalence for quadratic discrepancies]\label{prop:stein-godambe}
Let $\{P_\theta:\theta\in\Theta\subseteq\mathbb R^p\}$ be a parametric model,
and let $\psi(x;\theta)\in\mathbb R^m$ be a vector of functions such that
$\mathbb E_\theta[\psi(X;\theta)]=0$ for all $\theta\in\Theta$.
Assume that $\psi$ is a regular inference function, with sensitivity and
variability matrices
\[
S(\theta)
=
-\mathbb E_\theta[\partial_\theta \psi(X;\theta)],
\qquad
V(\theta)
=
\mathbb E_\theta[\psi(X;\theta)\psi(X;\theta)^\top],
\]
where $V(\theta)$ is positive definite and $S(\theta)$ is non-singular.

Define a Stein-type quadratic discrepancy by
\[
D_{\mathrm S}(Q,P_\theta)^2
=
\Psi(Q,\theta)^\top W(\theta)\Psi(Q,\theta),
\qquad
\Psi(Q,\theta)=\mathbb E_Q[\psi(X;\theta)],
\]
where $W(\theta)$ is a symmetric positive-definite matrix.

Then, for $h\to 0$,
\[
D_{\mathrm S}(P_{\theta+h},P_\theta)^2
=
h^\top S(\theta)^\top W(\theta) S(\theta)\, h
+o(\|h\|^2).
\]
In particular, if $W(\theta)=V(\theta)^{-1}$, then the induced Riemannian metric
coincides with the Godambe information,
$G_{\mathrm S}(\theta)=G(\theta)=S(\theta)^\top V(\theta)^{-1}S(\theta)$.
\end{proposition}

\begin{proof}
Since $\mathbb E_\theta[\psi(X;\theta)]=0$, a first-order expansion yields
\[
\Psi(P_{\theta+h},\theta)
=
\mathbb E_{\theta+h}[\psi(X;\theta)]
=
S(\theta)h+o(\|h\|),
\]
where we used Assumption~\ref{ass:weakC1}: the map $\eta \mapsto
\langle T_\eta, \psi(\cdot;\theta)\rangle$ is differentiable at $\eta
= \theta$ with derivative $\langle \partial_\theta T_\theta,
\psi(\cdot;\theta)\rangle$, which equals $S(\theta)$ by the weak
Bartlett identity \eqref{eq:weak-bartlett}.
Substituting into the quadratic form gives
\[
D_{\mathrm S}(P_{\theta+h},P_\theta)^2
=
(S(\theta)h)^\top W(\theta)(S(\theta)h)
+o(\|h\|^2),
\]
which proves the result.
\end{proof}

\begin{remark}\label{rem:stein-godambe}
This result shows that quadratic Stein discrepancies built from a finite
collection of Stein identities generate the same local geometry as inference
functions. Choosing $W(\theta)=V(\theta)^{-1}$ yields the optimal Godambe metric.
\end{remark}

\begin{remark}[The metric as a first fundamental form]\label{rem:first-fundamental-form}
Proposition~\ref{prop:stein-godambe} has a geometric reading that ties
the present construction to the transversality companion
\cite{LabouriauTransversality}. Collecting the instrument outputs
into a \emph{feature map}
$\Phi_\varphi(\theta) = ({}^{(\varphi)}m_{j}(\theta))_{j}$ from $\Theta$
into a feature space, the unweighted Godambe metric is the first
fundamental form of the induced immersion,
\[
G(\theta) = \bigl(D\Phi_\varphi(\theta)\bigr)^\top
            \bigl(D\Phi_\varphi(\theta)\bigr),
\]
and the optimal weighting $W(\theta)=V(\theta)^{-1}$ is the whitening of
the instrument outputs. Classical information geometry is the special
case in which the instrument is trivial ($\varphi\equiv 1$) and the
feature map is the score. In this language the regularity hypotheses of
Theorem~\ref{thm:main} are transversality conditions on $\Phi_\varphi$,
which hold for a generic instrument \cite{LabouriauTransversality}.
\end{remark}

\subsection{The examples revisited}\label{ssec:stein-examples}

Each example of Sections~\ref{sec:examples}--\ref{sec:spde}
instantiates Proposition~\ref{prop:stein-godambe} directly; we record
the outcomes.

\paragraph{Gaussian location.}
For $X \sim \mathcal N(\theta,\sigma^2)$ with the classical Stein
operator $\mathcal A_\theta g = g' - \frac{x-\theta}{\sigma^2} g$
(Section~\ref{ssec:weak-stein}), the choice $g \equiv 1$ gives $\psi
= -(x-\theta)/\sigma^2$ and $G_{\mathrm S}(\theta) = 1/\sigma^2 =
I(\theta)$: in the regular case, Stein, Fisher, and Godambe
geometries agree.

\paragraph{Uniform scale.}
The instrument $\psi = x - \theta/2$ with weight $V(\theta)^{-1}$
gives $D_{\mathrm S}(P_{\theta+h},P_\theta)^2 = 3h^2/\theta^2 +
o(h^2)$: the Stein metric is the Godambe metric
$G_{\mathrm S}(\theta) = 3/\theta^2$.

\paragraph{Cantor location.}
The instrument $\psi = x - \theta - \tfrac12$ gives
$G_{\mathrm S} \equiv 8$, and the transform residuals give
$G_{\mathrm S,t} = t^2\phi_Z(t)^2/(1-\phi_Z(t)^2)$ on the frequency
window of Section~\ref{ssec:cantor}: quadratic Stein geometry exists
where no likelihood does.

\paragraph{Stratified mixture.}
The stratification-agnostic pair $W = (\psi_a,\psi_b)^\top$ of
Proposition~\ref{prop:mixture} with weight $V(\theta)^{-1}$ recovers
$G_{\mathrm S}(\theta) = G(\theta)$ uniformly over the mixing matrix
$\Pi$.

\paragraph{Transform-based construction.}
For any model with computable characteristic function, the residuals
$\psi_t = e^{itx} - \phi_\theta(t)$ define quadratic Stein
discrepancies with the Hermitian weighting
$(1-|\phi_\theta(t)|^2)^{-1}$, inducing
\[
G_{\mathrm S,t}(\theta)
=
\frac{|\partial_\theta \phi_\theta(t)|^2}
{1-|\phi_\theta(t)|^2},
\]
a legitimate quadratic weighting by
Proposition~\ref{prop:stein-godambe}; the optimal weighting of the
two-dimensional real residual $(\mathrm{Re}\,\psi_t,
\mathrm{Im}\,\psi_t)$ in general differs. As $t \to 0$ these metrics
converge to the corresponding moment-based metrics ($3/\theta^2$ for
the uniform model, $8$ for the Cantor model): moment-based and
transform-based geometries are instances of one Stein-type
construction. The same recipe applies to the stratified mixture
through the real and imaginary parts of its characteristic-function
residuals.

\begin{remark}\label{rem:transform-moment}
The transform-based geometry need not coincide with the moment-based geometry.
The two constructions use different classes of identities: one is
finite-dimensional and algebraic, the other is spectral and depends on the
chosen frequencies. Thus, even within the same model, one may obtain different
weak Riemannian structures.
\end{remark}

\subsection{RKHS Stein discrepancies and the hierarchy of geometries}\label{ssec:rkhs}

We now consider Stein discrepancies defined over large classes of test functions,
in particular those arising from reproducing kernel Hilbert spaces (RKHS). These
discrepancies are widely used in goodness-of-fit testing
\cite{LiuLeeJordan2016,ChwalaSejdinovic2016,GorhamMackey2017} and in
minimum-discrepancy estimation \cite{Barp2019}, and provide a natural
framework for extending Stein's method beyond finite-dimensional
settings. From the inferential side, the aggregation of a continuum of
moment-type identities carried out here is the geometric counterpart
of the continuum generalised method of moments of Carrasco and Florens
\cite{CarrascoFlorens2000}, where the optimal weighting is a covariance
operator; the object of interest below is not efficiency but the
induced local geometry.

Let $\mathcal H$ be an RKHS with kernel $k(x,y)$, and let $\mathcal A_\theta$
be a (weak) Stein operator. The kernel Stein discrepancy is
\[
D_{\mathrm S}(Q,P_\theta)
=
\sup_{g\in\mathcal H,\ \|g\|_{\mathcal H}\le 1}
\left|
\mathbb E_Q[\mathcal A_\theta g(X)]
\right|.
\]
For each $g\in\mathcal H$, define $\psi_g(x;\theta)=\mathcal A_\theta g(x)$,
so that $\mathbb E_\theta[\psi_g(X;\theta)]=0$. To expand the
discrepancy along the model one needs the remainder to be uniform
over the unit ball of $\mathcal H$; pointwise expansions for each
fixed $g$ do not suffice. We therefore assume:

\begin{assumption}[Fr\'echet differentiability of the embedded model
map]\label{ass:frechet}
The map $\eta \mapsto \mu_\theta(\eta) \in \mathcal H^{*}$ defined by
$\mu_\theta(\eta)(g) := \mathrm{E}_\eta[\mathcal A_\theta g(X)]$ is
Fr\'echet differentiable at $\eta = \theta$, with derivative
$\mathcal S_\theta : \mathbb R^p \to \mathcal H^{*}$,
$(\mathcal S_\theta h)(g) = S_g(\theta)^\top h$, where
$S_g(\theta)=-\mathbb E_\theta[\partial_\theta \psi_g(X;\theta)]$.
\end{assumption}

Under Assumption~\ref{ass:frechet}, since $\mu_\theta(\theta)=0$,
\[
D_{\mathrm S}(P_{\theta+h},P_\theta)
= \bigl\| \mu_\theta(\theta+h) \bigr\|_{\mathcal H^{*}}
= \bigl\| \mathcal S_\theta h \bigr\|_{\mathcal H^{*}} + o(\|h\|)
= \sup_{\|g\|_{\mathcal H}\le 1}
\left| S_g(\theta)^\top h \right| + o(\|h\|),
\]
with the remainder uniform, and hence
$D_{\mathrm S}(P_{\theta+h},P_\theta)^2 = h^\top G_{\mathrm S}(\theta) h
+ o(\|h\|^2)$,
where $G_{\mathrm S}(\theta) = \mathcal S_\theta^* \mathcal S_\theta$.
Positive definiteness of $G_{\mathrm S}$ is not automatic: it requires
the Stein class to be rich enough that no parameter direction has
vanishing sensitivity against every $g \in \mathcal H$. This is a
transversality condition on the embedded model map, of exactly the
kind studied in the companion paper
\cite{LabouriauTransversality}.

In contrast with the finite-dimensional quadratic case, the RKHS Stein
discrepancy does not, in general, reduce to the Godambe metric associated
with a single inference function. Rather, it corresponds to an aggregation
over a class of inference functions indexed by $g\in\mathcal H$.

\begin{proposition}[Quadratic versus rich Stein geometries]\label{prop:hierarchy}
Let $\{P_\theta:\theta\in\Theta\subseteq\mathbb R^p\}$ be a parametric model.
\begin{enumerate}
    \item If a Stein discrepancy is defined through a finite-dimensional
    quadratic form of expectation residuals with positive-definite weight
    $W(\theta)$, then the local quadratic expansion induces the metric
    $G_{\mathrm S}(\theta)=S(\theta)^\top W(\theta)S(\theta)$.
    In particular, if $W(\theta)=V(\theta)^{-1}$, the Stein geometry
    coincides with the Godambe geometry.

    \item If the discrepancy is the RKHS supremum above and
    Assumption~\ref{ass:frechet} holds, the induced local quadratic
    form is $G_{\mathrm S}(\theta) = \mathcal S_\theta^{*}\mathcal
    S_\theta$, the Gram form of the sensitivity operator. This is a
    quadratic Stein metric whose weighting is determined by the RKHS
    geometry rather than by the variability of any finite probe
    vector; in general it coincides with the Godambe metric of no
    single finite-dimensional inference function, while the
    restriction of the discrepancy to any finite-dimensional subspace
    of $\mathcal H$ recovers a Godambe-type metric.
\end{enumerate}
\end{proposition}

\begin{proof}
\emph{Part~(1).} In the finite-dimensional quadratic case, the
inference function $\psi(x;\theta) \in \mathbb R^m$ satisfies
$\mathbb E_\theta[\psi(X;\theta)] = 0$, and a first-order expansion
around $\theta$ yields
$\Psi(P_{\theta+h},\theta) = S(\theta)h + o(\|h\|)$.
Substituting into the quadratic form
$D_{\mathrm S}(P_{\theta+h},P_\theta)^2
= \Psi^\top W(\theta) \Psi$
gives
\[
D_{\mathrm S}(P_{\theta+h},P_\theta)^2
= h^\top S(\theta)^\top W(\theta) S(\theta) h + o(\|h\|^2).
\]
This identifies the induced metric as
$G_{\mathrm S}(\theta) = S(\theta)^\top W(\theta) S(\theta)$.
The choice $W(\theta) = V(\theta)^{-1}$ recovers
$G_{\mathrm S}(\theta) = S(\theta)^\top V(\theta)^{-1} S(\theta)
= G(\theta)$.

\emph{Part~(2).} When the discrepancy is defined as
\[
D_{\mathrm S}(Q,P_\theta)
= \sup_{\|g\|_{\mathcal H} \leq 1}
\bigl|\mathrm{E}_Q[\mathcal A_\theta g(X)]\bigr| ,
\]
the squared discrepancy along the model expands as
$D_{\mathrm S}(P_{\theta+h},P_\theta)^2
= h^\top (\mathcal S_\theta^* \mathcal S_\theta) h + o(\|h\|^2)$,
where $\mathcal S_\theta$ maps $h$ to the functional
$g \mapsto S_g(\theta)^\top h$ on $\mathcal H$. The resulting metric
$G_{\mathrm S}(\theta) = \mathcal S_\theta^* \mathcal S_\theta$
aggregates the sensitivities of all test functions in $\mathcal H$.
For any finite-dimensional subspace
$\mathrm{span}\{g_1,\dots,g_m\} \subset \mathcal H$, the restriction
of the discrepancy to that subspace recovers a Godambe-type metric
associated with the inference function
$\psi = (\psi_{g_1},\dots,\psi_{g_m})^\top$. However, the full RKHS
metric need not equal any single such restriction, as it reflects the
supremum over the entire unit ball.
\end{proof}

\begin{remark}\label{rem:hierarchy}
This proposition shows that Godambe geometry occupies the finite-dimensional
quadratic core of a broader class of Stein-induced geometries. Stein
discrepancies extend inference-function geometry from individual estimating
equations to whole classes of model-characterising identities.
\end{remark}

The hierarchy may be summarised as follows:
finite-dimensional quadratic discrepancies based on a fixed set of identities
yield the Godambe geometry; parametric families of identities (e.g., indexed by
frequencies in transform-based constructions) yield families of metrics reflecting
different aspects of the model; and more general Stein discrepancies defined
through RKHS or supremum norms may give rise to geometric structures that are
not reducible to a single Godambe metric. This suggests a broader perspective on
information geometry, where one considers a family of admissible geometries
generated by different classes of identities or operators.

\section{Family of metrics and canonicity}\label{sec:canonicity}

The Fisher metric is canonical: by Chentsov's theorem it is, up to a
scale factor, the unique Riemannian metric invariant under Markov
morphisms---sufficient transformations of the data
\cite{Chentsov1982,AyJostLeSchwachhofer2017}. In the present framework
this uniqueness is lost, and for a structural reason: there is no
canonical \emph{instrument}. A Godambe metric is the geometry the
model displays when probed by a particular instrument (it depends,
for example, on the tuning parameter $c$ in the sinusoidal case), and
different instruments read different information about the same law.
One therefore obtains a \emph{family} of Riemannian metrics
$\{G_\psi(\theta)\}_{\psi \in \mathcal F}$ indexed by the class
$\mathcal F$ of admissible inference functions. The trade should be
stated plainly: what is given up is Chentsov--Markov invariance---a
Godambe metric is not, in general, preserved under sufficient
reductions of the data---and what is gained is \emph{existence}, on
models where no Markov-invariant metric is available because the
Fisher construction fails. Deliberately non-canonical families of
geometries on a single model have precedents, notably the
preferred-point geometries of Critchley, Marriott and Salmon
\cite{CritchleyMarriottSalmon1993}.

This is not unlike the situation with $\alpha$-connections in Amari's
theory \cite{Amari1985,AmariNagaoka2000}, where one has a family of
affine connections parametrised by~$\alpha$. Several natural strategies
for selecting a canonical metric within this family suggest themselves:
one may optimise over the inference function to obtain the supremal
Godambe information (analogous to the optimal inference function in the
Godambe theory \cite{Godambe1960}); take the metric obtained from the
optimal inference function, which, when the score exists, recovers the
Fisher metric; or study the entire family as a geometric object in its
own right.

\begin{remark}[No maximal element in non-regular models]\label{rem:no-max}
The first strategy is available exactly when the family
$\{G_\psi\}_{\psi\in\mathcal F}$ has a top element---which, by
Proposition~\ref{prop:loewner}, is the case when the score exists and
is admissible, the top being the Fisher metric. In the non-regular
models this paper targets, the supremum need not be attained, nor
even be finite. In the uniform scale model of
Section~\ref{ssec:uniform}, the endpoint probes
\[
\psi_c(x;\theta) = \mathbf 1\{x \le c\} - \frac{c}{\theta},
\qquad 0 < c < \theta
\]
(each a weak regular inference function on $(c,\infty)$, bounded, with
a $C^\infty$ smoothing of the indicator changing $S$ and $V$ by
arbitrarily little) have
\[
S(\theta) = -\frac{c}{\theta^2},
\qquad
V(\theta) = \frac{c}{\theta}\Bigl(1-\frac{c}{\theta}\Bigr),
\qquad
G_c(\theta) = \frac{c}{\theta^2(\theta - c)}
\;\longrightarrow\; \infty
\quad (c \uparrow \theta).
\]
Hence $\sup_{\psi} G_\psi(\theta) = +\infty$: there is no optimal
instrument and no Fisher-type bound to saturate. This is the
geometric face of superefficiency: the maximum-likelihood estimator
$\max_i X_i$ converges at rate $n^{-1}$, faster than the
$n^{-1/2}$-scale quantified by any single Godambe metric, and the
diverging endpoint probes trace exactly the information that the
$n^{-1/2}$-scale misses. The same phenomenon occurs in the shifted
exponential model (probes concentrating at the support endpoint) and,
in extreme form, in the undominated Cantor family
(Remark~\ref{rem:cantor-instrument}), where a single observation
already confines the parameter to a null set. A Godambe metric must
therefore be read as the geometry of the model \emph{as seen through
a particular instrument}; the canonicity question is not which metric
is true, but which instrument matches the inferential purpose
(Section~\ref{ssec:roles}). The contrast with the regular case is
instructive: there the family of Godambe metrics has a
supremum---the Fisher metric---approached but not attained by any
finite instrument family (Remark~\ref{rem:cauchy-closure}), whereas
here it has no upper bound at all.
\end{remark}

\subsection{Different roles of the metric}\label{ssec:roles}

The usefulness of a Riemannian structure in statistics is not limited to a single
purpose. It is helpful to distinguish several roles
(see also \cite{BarndorffNielsenCoxReid1986} for early discussion of the
geometric role, and \cite{AmariNagaoka2000} for the connection between
geometric structure and statistical inference).

\paragraph{(i) Inferential role.}
In classical theory, the information matrix is primarily used to quantify the
precision of estimators \cite{Godambe1960,BarndorffNielsen1978}. The inverse
of the metric determines the asymptotic covariance matrix, and confidence
regions are constructed from its quadratic form. In this context, optimality
is essential: one seeks an inference function that maximises the Godambe
information, or equivalently minimises asymptotic variance.

\paragraph{(ii) Diagnostic role.}
The metric also serves as a diagnostic tool for local properties of the model.
Near-singularity of the information matrix indicates weak identifiability, and
its eigenstructure reveals directions of strong and weak information
\cite{Watanabe2009}. These features are largely independent of optimality:
any regular metric reflecting the local behaviour of the model can provide
useful insight.

\paragraph{(iii) Geometric role.}
Beyond local considerations, the metric defines a global geometric structure on
the parameter space. Concepts such as geodesics and curvature describe how the
model behaves globally, how parameters interact, and how nonlinear the model is
\cite{Amari1985,AmariNagaoka2000}. These properties are not tied to efficiency,
and may be studied using any coherent Riemannian metric.

\paragraph{(iv) Computational role.}
Riemannian structures are increasingly used in computational statistics and
machine learning \cite{Amari1998}. Metrics can be used to precondition
optimisation algorithms, define natural gradient methods, or guide sampling
procedures. In these applications, stability and tractability are often more
important than optimality.

\subsection{Optimality versus coherence}\label{ssec:optimality}

These considerations suggest a fundamental distinction. When the goal is optimal
inference, the choice of metric should be guided by efficiency, leading to the
Fisher or optimal Godambe metric. When the goal is to understand the structure,
stability, or computation associated with a model, optimality is less important,
and any coherent metric may be used.

In the distributional framework, this distinction becomes particularly clear.
Different inference functions give rise to different Godambe metrics, and hence
to a family of admissible geometries on the same statistical model. Among these,
one may be selected for optimality, but the others remain valuable for
understanding the model from a geometric or computational perspective.

\subsection{Implications for non-regular models}\label{ssec:nonregular}

The extension to weak and distributional inference functions is especially
relevant in non-regular settings. In models with parameter-dependent support or
in finite mixtures, the likelihood-based score may fail to be a valid inference
function, and the Fisher information may not provide a meaningful metric.
Nevertheless, regular weak inference functions may still exist, leading to a
well-defined Godambe geometry.

This suggests that the existence of a Riemannian structure should not be tied to
the existence of a likelihood or of a regular score function. Instead, it should
be viewed as a consequence of the existence of a suitable inference function.

\section{Godambe--Riemannian structure and inferential separation}\label{sec:nonformation}

The classical notions of sufficiency, ancillarity, and their
interplay---collectively known as \emph{inferential separation}---play
a foundational role in statistical theory. These concepts, largely
developed by R.A.\ Fisher and formalised by Basu, Barndorff-Nielsen
\cite{BarndorffNielsen1978} and others, provide the theoretical
justification for data reduction and conditioning arguments. The unified
framework of \emph{nonformation}, due to Barndorff-Nielsen and further
developed in J{\o}rgensen and Labouriau
\cite{JorgensenLabouriau2012}, provides a single principle that
encompasses both sufficiency and ancillarity as special cases. Weak
versions of these concepts, adapted to the distributional setting, are
introduced in \cite{C}.

In this section we investigate how the Godambe--Riemannian structure
interacts with weak inferential separation. The main observation is
that the block structure of the Godambe metric---specifically, the
vanishing of its off-diagonal (interest--nuisance) block---provides
the geometric expression of inferential separation in the
inference-function framework, and that the weak Bartlett identity
\eqref{eq:weak-bartlett} ties this block structure to a nuisance
tangent space that lives in $\mathcal S'(\mathbb R^k)$, requiring
neither densities nor square-integrable scores.

Throughout this section, following the companion paper
\cite[Section~8]{C}, we write $\theta = (\alpha, \beta)$, where
$\alpha \in A \subseteq \mathbb R^{a}$ is the parameter of interest
and $\beta \in B \subseteq \mathbb R^{b}$ is the nuisance parameter;
this use of $\alpha$ and $\beta$ is local to the present section and
bears no relation to the stability index and eigenmode scales of
Section~\ref{sec:spde}. We consider a joint inference function
$\psi_\theta = (\psi_\alpha, \psi_\beta)$, where $\psi_\alpha$ is an
inference function for $\alpha$ (possibly depending on $\beta$) and
$\psi_\beta$ one for $\beta$, and write the block decompositions
\[
S(\theta) =
\begin{pmatrix}
S_{\alpha\alpha} & S_{\alpha\beta} \\
S_{\beta\alpha} & S_{\beta\beta}
\end{pmatrix},
\qquad
V(\theta) =
\begin{pmatrix}
V_{\alpha\alpha} & V_{\alpha\beta} \\
V_{\beta\alpha} & V_{\beta\beta}
\end{pmatrix},
\qquad
G(\theta) =
\begin{pmatrix}
G_{\alpha\alpha} & G_{\alpha\beta} \\
G_{\beta\alpha} & G_{\beta\beta}
\end{pmatrix},
\]
with $S_{\alpha\beta} = -\mathrm{E}_\theta[\partial_\beta
\psi_\alpha]$, $V_{\alpha\beta} =
\mathrm{E}_\theta[\psi_\alpha\psi_\beta^\top]$, and $G = S^\top V^{-1}
S$ as before.

\subsection{The nonformation principle}\label{ssec:nonformation-principle}

The \emph{nonformation principle}
(Barndorff-Nielsen, 1978) states that when a sub-model obtained by
fixing a statistic $U = u(X)$ does not contain information about $\alpha$,
inference about $\alpha$ should be based on the marginal distribution of
$U$ alone (sufficiency) or on the conditional distribution given $U$
(ancillarity).

The precise meaning of ``does not contain information'' gives rise to
different notions of nonformation:
\emph{S-nonformation} requires a likelihood factorisation where the
conditional factor does not depend on $\alpha$;
\emph{I-nonformation} requires that the conditional sub-model be
saturated; and
\emph{L-nonformation} provides a unifying framework through the
profile likelihood.
These concepts satisfy a natural hierarchy
(J{\o}rgensen and Labouriau, 2012, Chapter~3): B-sufficiency implies
S-sufficiency, which implies L-sufficiency; and G-sufficiency also
implies L-sufficiency.

In the distributional framework, weak versions of these notions are
defined without requiring densities
(see \cite{C}, Section~8.7). Weak S-nonformation, for
instance, requires that
$\mathrm{E}_\theta[g(X) \mid U]$ does not depend on $\alpha$ for all
test functions $g \in \mathcal S(\mathbb R^k)$, rather than requiring a
density-based factorisation.

\subsection{Orthogonality and nonformation in the Godambe metric}\label{ssec:orthogonality-nonformation}

The connection between the Godambe--Riemannian structure and
inferential separation rests on a nuisance tangent space that is
defined \emph{on the model side}, in $\mathcal S'(\mathbb R^k)$,
following the two-stage construction of \cite[Section~8.3]{C}: the
distributional stage defines the nuisance directions without densities;
the Hilbert-space stage, when available, supplies inner products via
Riesz representatives. Only the distributional stage is needed for the
geometry.

\begin{definition}[Weak nuisance tangent space; Godambe orthogonality]
\label{def:godambe-orthogonality}
Under Assumption~\ref{ass:weakC1}, the \emph{weak nuisance tangent
space} at $\theta = (\alpha,\beta)$ is
\[
\mathcal T_N(\theta)
\;=\; \mathrm{span}\bigl\{ \partial_{\beta_j} T_\theta
      : j = 1,\dots,b \bigr\}
\;\subset\; \mathcal S'(\mathbb R^k).
\]
An inference function $\psi_\alpha(\cdot\,;\theta)$ for $\alpha$ is
\emph{Godambe-orthogonal} to the nuisance parameter $\beta$ at
$\theta$ if
\begin{equation}\label{eq:godambe-orthogonality}
\bigl\langle \dot T,\; \psi_{\alpha,i}(\cdot\,;\theta) \bigr\rangle
= 0
\qquad
\text{for every } \dot T \in \mathcal T_N(\theta)
\text{ and } i = 1,\dots,a .
\end{equation}
\end{definition}

When densities exist and the nuisance scores are square-integrable,
the pairing in \eqref{eq:godambe-orthogonality} equals
$\mathrm{E}_\theta[\psi_{\alpha,i}\,\partial_{\beta_j}\!\log
p_\theta]$ and Definition~\ref{def:godambe-orthogonality} reduces to
the classical orthogonality to the span of the nuisance scores; the
weak formulation, however, survives models with no density at
all---such as the Cantor location--scale family of
Section~\ref{ssec:cantor}---for which no score, square-integrable or
otherwise, is available. The link with the block
structure of Section~\ref{ssec:rif-def} is the weak Bartlett identity.

\begin{lemma}[Weak Bartlett identity, nuisance block]\label{lem:bartlett-nuisance}
Under Assumption~\ref{ass:weakC1}, differentiation of the unbiasedness
identity $\langle T_{\alpha,\beta}, \psi_\alpha(\cdot\,;
\alpha,\beta)\rangle = 0$ with respect to $\beta_j$ gives
\[
S_{\alpha\beta}(\theta)_{\cdot j}
\;=\; -\,\mathrm{E}_\theta\bigl[\partial_{\beta_j}\psi_\alpha\bigr]
\;=\; \bigl\langle \partial_{\beta_j} T_\theta,\;
      \psi_\alpha(\cdot\,;\theta) \bigr\rangle .
\]
Consequently, $\psi_\alpha$ is Godambe-orthogonal to $\beta$ if and
only if its cross-sensitivity vanishes, $S_{\alpha\beta}(\theta) = 0$.
In particular, if $\psi_\alpha$ does not depend on $\beta$ and is
unbiased uniformly in $\beta$, then $S_{\alpha\beta} \equiv 0$
automatically (cf.\ \cite[Proposition~8.1]{C}).
\end{lemma}

\begin{proof}
$0 = \partial_{\beta_j} \langle T_\theta, \psi_\alpha\rangle
= \langle \partial_{\beta_j}T_\theta, \psi_\alpha\rangle
+ \langle T_\theta, \partial_{\beta_j}\psi_\alpha\rangle$ by
Assumption~\ref{ass:weakC1}; rearrange. If $\psi_\alpha$ is free of
$\beta$, the second term vanishes term by term and the first equals
$\partial_{\beta_j}$ of a constant.
\end{proof}

The geometric meaning of the block structure is the content of the
next proposition. Vanishing of $G_{\alpha\beta}$ is precisely
Riemannian orthogonality of the coordinate splitting $\theta =
(\alpha,\beta)$ in the metric $G$.

\begin{proposition}[Block structure and inferential separation]
\label{prop:orthogonality-nonformation}
Let the joint inference function $\psi_\theta = (\psi_\alpha,
\psi_\beta)$ satisfy the hypotheses of
Proposition~\ref{prop:godambe-metric}, and let $\hat\theta_n =
(\hat\alpha_n, \hat\beta_n)$ solve the joint estimating equation.
Then:
\begin{enumerate}
\item The asymptotic covariance of $\sqrt n(\hat\alpha_n - \alpha)$
is $[G(\theta)^{-1}]_{\alpha\alpha} = \bigl(G_{\alpha\alpha} -
G_{\alpha\beta}G_{\beta\beta}^{-1}G_{\beta\alpha}\bigr)^{-1} \succeq
G_{\alpha\alpha}^{-1}$, with equality if and only if
$G_{\alpha\beta}(\theta) = 0$. Thus the vanishing of the off-diagonal
block of the Godambe metric is exactly the condition under which the
presence of the nuisance parameter costs nothing asymptotically:
inference about $\alpha$ proceeds as if $\beta$ were known.
\item If $S_{\alpha\beta} = 0$, $S_{\beta\alpha} = 0$ and
$V_{\alpha\beta} = 0$, then $G_{\alpha\beta} = 0$.
\item Godambe orthogonality alone ($S_{\alpha\beta} = 0$,
equivalently \eqref{eq:godambe-orthogonality}) does not in general
imply $G_{\alpha\beta} = 0$; its operational content is
\emph{nuisance insensitivity}: first-order perturbations of $\beta$,
and in particular the plug-in of a consistent estimator
$\hat\beta_n$, do not affect the asymptotic behaviour of the
estimating equation for $\alpha$.
\end{enumerate}
\end{proposition}

\begin{proof}
(1) The Schur-complement formula for the inverse of the partitioned
matrix $G$ gives $[G^{-1}]_{\alpha\alpha} = (G_{\alpha\alpha} -
G_{\alpha\beta}G_{\beta\beta}^{-1}G_{\beta\alpha})^{-1}$, and
$G_{\alpha\beta}G_{\beta\beta}^{-1}G_{\beta\alpha} \succeq 0$ with
equality if and only if $G_{\alpha\beta} = 0$; matrix inversion
reverses the Loewner order. The identification of
$[G^{-1}]_{\alpha\alpha}$ as the asymptotic covariance is the standard
sandwich argument for the joint equation, using
Assumption~\ref{ass:weakC1} for the expansion.
(2) With $V$ block-diagonal, $V^{-1} =
\mathrm{diag}(V_{\alpha\alpha}^{-1}, V_{\beta\beta}^{-1})$ and
\[
G_{\alpha\beta}
= S_{\alpha\alpha}^\top V_{\alpha\alpha}^{-1} S_{\alpha\beta}
+ S_{\beta\alpha}^\top V_{\beta\beta}^{-1} S_{\beta\beta}
= 0 + 0 .
\]
(3) The displayed formula in (2) shows that even with
$V_{\alpha\beta}=0$ and $S_{\alpha\beta}=0$ the term
$S_{\beta\alpha}^\top V_{\beta\beta}^{-1} S_{\beta\beta}$ survives
unless the nuisance equation is also insensitive to $\alpha$. The
plug-in statement follows from the expansion of the $\alpha$-equation
in $(\alpha, \beta)$: the $\beta$-derivative of its expectation is
$-S_{\alpha\beta} = 0$.
\end{proof}

\subsection{The Bhapkar--Godambe projection}\label{ssec:bhapkar-godambe-nonformation}

When separation does not hold automatically, part of it can always be
manufactured by projection. The \emph{Bhapkar--Godambe projection}
(see \cite{Labouriau1996,C}) replaces $\psi_\alpha$ by its residual on
the span of the nuisance inference function:
\[
\psi_\alpha^*(\cdot\,;\theta)
= \psi_\alpha(\cdot\,;\theta)
- V_{\alpha\beta}(\theta)\, V_{\beta\beta}(\theta)^{-1}\,
  \psi_\beta(\cdot\,;\theta),
\]
the orthogonal projection of $\psi_\alpha$ onto the orthocomplement of
$\mathrm{span}\{\psi_{\beta,1},\dots,$ $\psi_{\beta,b}\}$ in the Hilbert
space of inference functions with inner product $\langle f, g
\rangle_\theta = \mathrm{E}_\theta[f g^\top]$.

\begin{proposition}[What the projection does and does not achieve]
\label{prop:bg-projection}
The adjusted inference function $\psi_\alpha^*$ satisfies:
\begin{enumerate}
\item $V^*_{\alpha\beta} :=
\mathrm{E}_\theta[\psi_\alpha^*\psi_\beta^\top] = 0$ always
(variability-orthogonality);
\item $S^*_{\alpha\beta} = S_{\alpha\beta} - V_{\alpha\beta}
V_{\beta\beta}^{-1} S_{\beta\beta}$; in particular the projection
achieves Godambe orthogonality ($S^*_{\alpha\beta} = 0$) if and only
if $S_{\alpha\beta} = V_{\alpha\beta}V_{\beta\beta}^{-1}S_{\beta\beta}$;
\item a sufficient condition for (2) is that $\psi_\beta$ \emph{weakly
generates the nuisance directions}: there is a nonsingular
$M(\theta)$ such that
$\langle \partial_{\beta_j} T_\theta, g\rangle = \sum_m M_{jm}(\theta)\,
\mathrm{E}_\theta[g\,\psi_{\beta,m}]$ for all admissible $g$. Then
$S_{\alpha\beta} = V_{\alpha\beta}M^\top$ and $S_{\beta\beta} =
V_{\beta\beta}M^\top$, whence $S^*_{\alpha\beta} = 0$: the projection
achieves full Godambe orthogonality. Classically this is the case
$\psi_\beta = \partial_\beta \log p_\theta$ (nuisance score), with $M
= I$.
\end{enumerate}
\end{proposition}

\begin{proof}
(1) is the defining property of the $L^2$-projection. For (2),
$\partial_\beta \psi_\alpha^* = \partial_\beta\psi_\alpha -
\partial_\beta\bigl(V_{\alpha\beta}V_{\beta\beta}^{-1}\bigr)\psi_\beta
- V_{\alpha\beta}V_{\beta\beta}^{-1}\partial_\beta\psi_\beta$; taking
$\mathrm{E}_\theta$, the middle term vanishes because
$\mathrm{E}_\theta[\psi_\beta] = 0$, giving $S^*_{\alpha\beta} =
S_{\alpha\beta} - V_{\alpha\beta}V_{\beta\beta}^{-1}S_{\beta\beta}$.
For (3), apply the generating identity to $g = \psi_{\alpha,i}$ and
$g = \psi_{\beta,i}$: the first gives $S_{\alpha\beta} =
V_{\alpha\beta}M^\top$ by Lemma~\ref{lem:bartlett-nuisance}, the
second $S_{\beta\beta} = V_{\beta\beta}M^\top$ (weak Bartlett applied
to $\psi_\beta$); substitute into (2).
\end{proof}

Geometrically: the projection always makes the \emph{variability}
block-diagonal; it makes the \emph{metric} block-diagonal exactly when
the nuisance inference function is rich enough to represent the
nuisance directions of the model, in the sense of (3).

\subsection{Automatic orthogonality in symmetric location-scale models}\label{ssec:automatic-orthogonality}

A remarkable feature of the distributional framework is that
symmetric location--scale models exhibit \emph{automatic}
separation for odd/even instrument pairs, with no projection needed
---including heavy-tailed families such as the Cauchy, for
which moment-based instruments are unavailable, and the Cantor
location--scale family, for which no score exists at all.

\begin{proposition}[Automatic block-diagonality for odd/even probes]
\label{prop:automatic-orthogonality}
Let $\mathcal P = \{P_{\mu,\sigma} : \mu \in \mathbb R, \sigma > 0\}$
be a location-scale family, $X = \mu + \sigma Z$, with $Z$ symmetric
and characteristic function $\phi_Z$ real and even. Take $\alpha =
\mu$ (interest), $\beta = \sigma$ (nuisance), and the bounded
instruments
\[
\psi_\mu(x;\mu) = \sin\bigl(c(x-\mu)\bigr),
\qquad
\psi_\sigma(x;\mu,\sigma) = \cos\bigl(d(x-\mu)\bigr) - \phi_Z(d\sigma),
\]
with $c, d \neq 0$ such that $\phi_Z(c\sigma) \neq 0$ and $\phi_Z$ is
differentiable at $d\sigma$ with $\phi_Z'(d\sigma) \neq 0$. Then, at
every $(\mu,\sigma)$:
\[
S_{\mu\sigma} = 0, \qquad
S_{\sigma\mu} = 0, \qquad
V_{\mu\sigma} = 0,
\]
hence $G_{\mu\sigma} = 0$ by
Proposition~\ref{prop:orthogonality-nonformation}(2): the Godambe
metric of the pair $(\psi_\mu,\psi_\sigma)$ is block-diagonal, and
the pair achieves exact inferential separation between location and
scale, for \emph{every} symmetric location-scale family.
\end{proposition}

\begin{proof}
All three identities are expectations of odd or even functions under
the symmetric law of $X - \mu = \sigma Z$, computed from the real CF.
(i) $S_{\mu\sigma}$: $\psi_\mu$ does not depend on $\sigma$ and is
unbiased uniformly in $\sigma$, since $\mathrm{E}[\sin(c(X-\mu))] =
\mathrm{Im}\,\phi_Z(c\sigma) = 0$; Lemma~\ref{lem:bartlett-nuisance}
gives $S_{\mu\sigma} = \langle \partial_\sigma T_{\mu,\sigma},
\psi_\mu\rangle = \partial_\sigma\, \mathrm{Im}\,\phi_Z(c\sigma) = 0$.
(ii) $S_{\sigma\mu} = -\mathrm{E}[\partial_\mu\psi_\sigma] =
-\,d\,\mathrm{E}[\sin(d(X-\mu))] = -\,d\,\mathrm{Im}\,\phi_Z(d\sigma)
= 0$.
(iii) $V_{\mu\sigma} = \mathrm{E}[\sin(c(X-\mu))\cos(d(X-\mu))] -
\phi_Z(d\sigma)\,\mathrm{E}[\sin(c(X-\mu))]$; by the product-to-sum
identity the first term is $\tfrac12\,\mathrm{Im}\,[\phi_Z((c+d)\sigma)
+ \phi_Z((c-d)\sigma)] = 0$, and the second vanishes as in (i).
Finally $S_{\mu\mu} = c\,\phi_Z(c\sigma) \neq 0$ and $S_{\sigma\sigma}
= d\,\phi_Z'(d\sigma) \neq 0$ by hypothesis, and $V$ is positive
definite (the two residuals are non-proportional bounded functions
under a law of full support), so
Proposition~\ref{prop:godambe-metric} applies and
Proposition~\ref{prop:orthogonality-nonformation}(2) yields
$G_{\mu\sigma} = 0$.
\end{proof}

In the language of nonformation, this automatic orthogonality reflects
the \emph{G-nonformation} structure of the location-scale model
(J{\o}rgensen and Labouriau, 2012, Section~3.3; see also
\cite[Section~8.7]{C}): the model is generated by the affine group,
and odd/even instrument pairs respect the group structure without any
adjustment. The Godambe information of the location component,
$G_{\mu\mu} = 2c^2\phi_Z(c\sigma)^2/(1-\phi_Z(2c\sigma))$, is the
closed-form expression obtained in \cite[Section~8.5]{C}. In
particular, Proposition~\ref{prop:automatic-orthogonality} applies
verbatim to the Cantor location--scale family of
Section~\ref{ssec:cantor} (Remark~\ref{rem:cantor-locscale}): exact
location--scale separation in a model that possesses no likelihood.

\subsection{Geometric interpretation}\label{ssec:geometric-nonformation}

The results above give a unified geometric picture. The Godambe
metric on the full parameter space $\Theta$ encodes both the information
about each parameter component and the interaction between them. The
off-diagonal block of the Godambe metric measures the
failure of inferential separation
(Proposition~\ref{prop:orthogonality-nonformation}): when it vanishes,
the interest and nuisance directions are Riemannian-orthogonal, and
inference about $\alpha$ proceeds, asymptotically, as if $\beta$ were
known.

The Bhapkar--Godambe projection provides a constructive mechanism that
always block-diagonalises the \emph{variability}, and
block-diagonalises the \emph{metric} exactly when the nuisance
inference function weakly generates the nuisance directions of the
model (Proposition~\ref{prop:bg-projection}). In the Riemannian
picture, this amounts to choosing an adapted basis of inference
functions in which the metric is block-diagonal.

This interpretation extends naturally to the distributional setting.
When classical densities are not available, the orthogonality
condition~\eqref{eq:godambe-orthogonality} is verified through
distributional expectations---involving characteristic functions,
distributional moments, or Schwartz-space pairings---rather than
through density-based integrals. The Godambe metric and its
block structure provide a purely inference-function-based
characterisation of nonformation that does not require a likelihood.

\section{Discussion}\label{sec:discussion}

The question this paper set out to answer is representational: how
much of the differential-geometric treatment of statistical models
survives when the density-based representation of information is
withdrawn. The answer is: essentially all of its local core, provided
the metric is allowed to come from an instrument rather than from the
likelihood. The construction reads simply---the law is the tempered
distribution $T_\theta$; an instrument extracts information from it;
the Godambe information is the geometry this information induces on
$\Theta$. Theorem~\ref{thm:main} attaches a smooth Riemannian metric
to any model so equipped; Proposition~\ref{prop:loewner} orders the
resulting family below the Fisher metric whenever the latter exists,
with the score attaining the top; and the examples of
Sections~\ref{sec:examples}--\ref{sec:spde} show that the hypotheses
are satisfiable far outside the Fisher--Rao class: under
parameter-dependent support, in the absence of any dominating measure
(the Cantor family, Proposition~\ref{prop:cantor-undominated}),
without moments of any order (the $\alpha$-stable lattice field), and
in dominated models whose score is a biased inference function (the
stratified mixtures of \cite{Labouriau2022}). The companion
inference-function paper \cite{C} supplies the instruments and their
efficiency theory; the transversality companion
\cite{LabouriauTransversality} shows that a generic instrument makes
the hypotheses of Theorem~\ref{thm:main} hold outside exceptional
configurations.

The lattice stochastic heat equation shows that the construction is
not confined to textbook families. The weak Godambe information of
eigenmode probes is available in closed form, the exact multi-probe
information follows from the joint characteristic function (with the
Gaussian decoupling recovered at $\alpha = 2$), and the geometry
carries a two-sided stability dictionary: it stabilises at rate
$\alpha\theta\lambda_k$ under stable dynamics
(Proposition~\ref{prop:spde-stability}), collapses under unstable
dynamics (Remark~\ref{rem:spde-collapse}), and its
frequency-optimised stationary form is the scale-invariant metric
$\mathrm{const}\cdot d\theta^2/\theta^2$
(Remark~\ref{rem:spde-adaptive}). Dynamical and geometric stability
are thus linked by explicit rates, without a variance in sight.

The price of the extension should be stated as plainly as the gain.
A Godambe metric is instrument-relative: by Chentsov's theorem no
Markov-invariant metric can exist at this level of generality
\cite{Chentsov1982,AyJostLeSchwachhofer2017}, and in non-regular
models the family of Godambe metrics has no maximal element
(Remark~\ref{rem:no-max})---the geometric face of superefficiency. We
regard this relativity not as a defect but as the accurate geometry
of measurement: a metric quantifies the information a given
instrument extracts; the family over instruments quantifies the
model; and the roles distinguished in Section~\ref{sec:canonicity}
(inferential, diagnostic, geometric, computational) select different
members for different purposes. Within this family, quadratic Stein
discrepancies built on finitely many identities read exactly the
Godambe geometry, while the reproducing-kernel constructions of
Section~\ref{sec:stein} aggregate a continuum of identities into
geometries beyond the reach of any single inference function, at the
cost of a Fr\'echet differentiability assumption and a nondegeneracy
condition of transversality type, inherited from and studied in the
companion papers \cite{D,LabouriauTransversality}.

The geometric reading of inferential separation completes the
picture. The nuisance tangent space lives on the model side, in
$\mathcal S'(\mathbb R^k)$; the weak Bartlett identity converts
pairing-orthogonality into the vanishing of a sensitivity block; and
full separation is exactly block-diagonality of the metric---achieved
constructively by the Bhapkar--Godambe projection under an explicit
generating condition
(Proposition~\ref{prop:bg-projection}), and automatically, for
odd/even instrument pairs, in every symmetric location--scale family
(Proposition~\ref{prop:automatic-orthogonality}), including the
heavy-tailed Cauchy family and the Cantor family, the latter
possessing no score at all. Data reduction
thereby acquires a metric expression that requires neither densities
nor square-integrable scores.

Beyond the directions already noted---curvature in multiparameter
families, $\alpha$-connection analogues with their duality and
flatness theory, and Godambe-based natural-gradient methods in
non-regular models---the Cantor example points to a broader
programme: models whose parameters sit inside the generating
iterated-function system (Remark~\ref{rem:cantor-selfsimilar}), so
that the parameter describes the geometric mechanism producing the
law rather than the law itself. The weak estimating equations of the
present framework extend to that setting, and we intend to pursue it
separately.

\bigskip

\noindent
\textbf{Acknowledgement.}
The author is grateful to Serhii V.~Zabolotnii for suggesting some
corrections in the first arXiv version of this paper regarding the role 
of the score of the Cauchy location--scale family 
(Sections~\ref{ssec:orthogonality-nonformation}
and~\ref{ssec:automatic-orthogonality}, Section~\ref{sec:discussion},
and Remark~\ref{rem:cauchy-sinusoid}). The Cauchy family is
likelihood-regular; what it lacks is moments. The multi-frequency
closure calculation of Remark~\ref{rem:cauchy-closure}---the explicit
Gram system, the optimised efficiencies $\kappa_m$, and the
completeness argument showing $\kappa_m\uparrow1$---is likewise due to
him (personal communication, July 2026), and is reproduced here with
his kind permission.

\end{document}